\documentclass[11pt]{amsart}
\RequirePackage{amsmath,amssymb,amsthm, amscd, comment,mathtools}

\usepackage{kantlipsum}
\usepackage[all]{xy}
\usepackage{mathabx}
\usepackage{graphicx}
\usepackage{tikz-cd}

\usepackage[colorlinks,linkcolor=red,anchorcolor=red,citecolor=blue]{hyperref}

\calclayout

\newtheorem{theorem}[subsubsection]{Theorem}
\newtheorem{proposition}[subsubsection]{Proposition}
\newtheorem{corollary}[subsubsection]{Corollary}
\newtheorem{lemma}[subsubsection]{Lemma}
\newtheorem{conjecture}[subsubsection]{Conjecture}

\theoremstyle{definition}

\newtheorem{definition}[subsubsection]{Definition}
\newtheorem{remark}[subsubsection]{Remark}

\newtheorem{example}[subsubsection]{Example}
\numberwithin{equation}{subsubsection}

\newcommand{\dimtot}{\mathrm{dimtot}}
\newcommand{\rank}{\mathrm{rank}}

\newcommand{\bQl}{\overline{\mathbb{Q}}_\ell}

\begin{document}

\title[relative twist formula]{On the relative twist formula of $\ell$-adic sheaves}

\author{Enlin Yang}
\address{Fakult\"at f\"ur Mathematik, Universit\"at Regensburg, 93040 Regensburg, Germany}
\email{enlin.yang@mathematik.uni-regensburg.de, yangenlin0727@126.com}

\author{Yigeng Zhao}
\address{Fakult\"at f\"ur Mathematik, Universit\"at Regensburg, 93040 Regensburg, Germany}
\email{yigeng.zhao@mathematik.uni-regensburg.de}

\date{\today}

\begin{abstract}
We propose a conjecture on the relative twist formula of $\ell$-adic sheaves, which can be viewed as a generalization of Kato-Saito's conjecture. We verify this conjecture under some transversal assumptions.

We also define a relative cohomological characteristic class and prove that its formation is compatible with proper push-forward.
A conjectural relation is also given between the relative twist formula and the relative cohomological characteristic class.
\end{abstract}

\maketitle

\tableofcontents

\section{Introduction}

As an analogy of the theory of $D$-modules, Beilinson \cite{Bei16} and T.~Saito \cite{Sai16} define the singular support and the characteristic cycle of an $\ell$-adic sheaf on a smooth variety respectively. As an application of their  theory, we prove a twist formula of epsilon factors in \cite{UYZ}, which is a modification of a conjecture due to Kato and T.~Saito\cite[Conjecture 4.3.11]{KS08}.
\subsection{Kato-Saito's conjecture}
\subsubsection{}
Let $X$ be a smooth projective scheme purely of dimension $d$
over a finite field $k$ of characteristic $p$.  
Let $\Lambda$ be a finite field of characteristic $\ell\neq p$ or $\Lambda=\overline{\mathbb Q}_\ell$.
Let  $\mathcal F\in D_c^b(X,\Lambda)$ and $\chi(X_{\bar{k}},\mathcal F)$ be the Euler-Poincar\'e characteristic of $\mathcal F$.
The Grothendieck  $L$-function $L(X,\mathcal F, t)$ satisfies the following functional equation
\begin{equation}\label{eqYZ:fe}
L(X,\mathcal F, t)=\varepsilon(X,\mathcal F)\cdot t^{-\chi(X_{\bar{k}},\mathcal F)}\cdot L(X, D(\mathcal F),t^{-1}),
\end{equation}
where $D(\mathcal F)$ is the Verdier dual  $R\mathcal Hom(\mathcal F,Rf^!\Lambda)$ of $\mathcal F$,  $f\colon X\rightarrow {\rm Spec}k$ is the structure morphism and
\begin{equation}
\varepsilon(X,\mathcal F)=\det(-\mathrm{Frob}_k;R\Gamma(X_{\bar{k}},\mathcal F))^{-1}
\end{equation}
is the epsilon factor (the constant term of the functional equation (\ref{eqYZ:fe})) and $\mathrm{Frob}_k$ is the geometric Frobenius (the inverse of the Frobenius substitution).

\subsubsection{}
In (\ref{eqYZ:fe}), both $\chi(X_{\bar{k}},\mathcal F)$ and $\varepsilon(X,\mathcal F)$ are related to ramification theory. 
Let $cc_{X/k}(\mathcal F)=0_X^!(CC(\mathcal F, X/k))\in CH_0(X)$ be the characteristic class of $\mathcal F$ (cf. \cite[Definition 5.7]{Sai16}), where $0_X\colon X\to T^\ast X$ is the zero section and $CC(\mathcal F, X/k)$ is the characteristic cycle of $\mathcal F$.
Then $\chi(X_{\bar{k}},\mathcal F)={\rm deg} (cc_{X/k}(\mathcal F))$ by the index formula \cite[Theorem 7.13]{Sai16}. 
The following theorem proved in \cite{UYZ} gives a relation between  $\varepsilon(X,\mathcal F)$ and $cc_{X/k}(\mathcal F)$, which is a modified version of the formula conjectured by Kato and T. Saito in \cite[Conjecture 4.3.11]{KS08}.
\begin{theorem}[Twist formula, {\cite[Theorem 1.5]{UYZ}}]\label{thm:uyz}
We have
\begin{equation}\label{eqYZ:ep}
\varepsilon(X,\mathcal F\otimes\mathcal G)=\varepsilon(X,\mathcal F)^{{\rm rank}\mathcal G}\cdot \det\mathcal G(\rho_X(-cc_{X/k}(\mathcal F)))\qquad{\rm in}~\Lambda^\times,
\end{equation}
where  $\rho_X\colon CH_0(X)\to\pi_1^{\rm ab}(X)$ is the reciprocity map defined by sending the class $[s]$ of a closed point $s\in X$ to the geometric Frobenius $\mathrm{Frob}_s$ and $\det\mathcal G\colon \pi_1^{\rm ab}(X)\to \Lambda^\times$ is the representation associated to the smooth sheaf $\det\mathcal G$ of rank 1.

\end{theorem}
When $\mathcal F$ is the constant sheaf $\Lambda$, this is proved by S. ~Saito \cite{SS84}. If $\mathcal F$ is a  smooth sheaf  on an open dense subscheme $U$ of $X$ such that $\mathcal F$ is tamely ramified along $D=X\setminus U$,  then Theorem \ref{thm:uyz} is a consequence of \cite[Theorem 1]{Sai93}.
In \cite{Vi09a,Vi09b}, Vidal proves a similar result on a proper smooth surface over a finite field of characteristic $p>2$ under certain technical assumptions. Our proof of Theorem \ref{thm:uyz} is based on the following theories: one is the theory of singular support \cite{Bei16} and  characteristic cycle \cite{Sai16}, and another is Laumon's product formula \cite{Lau87}. 

 \subsection{$\varepsilon$-factorization}
 \subsubsection{}Now we assume that $X$ is a smooth projective geometrically connected curve of genus $g$ over a finite field $k$ of characteristic $p$. Let $\omega$ be a non-zero rational $1$-form on $X$ and $\mathcal F$ an $\ell$-adic sheaf on $X$.
The following formula is conjectured by Deligne and proved by Laumon \cite[3.2.1.1]{Lau87}:
\begin{equation}\label{eqYZ:product}
\varepsilon(X,\mathcal F)=p^{[k:\mathbb F_p](1-g)\rank (\mathcal F)}\prod_{v\in|X|}\varepsilon_v(\mathcal F|_{X_{(v)}},\omega).
\end{equation}
For higher dimensional smooth scheme $X$ over $k$,
it is still an open question whether there is an   $\varepsilon$-factorization formula (respectively a geometric  $\varepsilon$-factorization formula) for $\varepsilon(X,\mathcal F)$ (respectively $\det R\Gamma(X,\mathcal F)$).

\subsubsection{}
In \cite{Bei07},  Beilinson develops the theory of topological epsilon factors using $K$-theory spectrum and he asks  whether his construction  admits a motivic ($\ell$-adic or de Rham) counterpart.
For de Rham cohomology, such a construction is given by  Patel in \cite{Pat12}. Based on \cite{Pat12},  Abe and Patel prove a similar twist formula in \cite{AP17} for global de Rham epsilon factors in the classical setting of $\mathcal D_X$-modules on smooth projective varieties over a field of characteristic zero.  In the $\ell$-adic situation, such a geometric $\varepsilon$-factorization formula is still open even if  $X$ is a curve. Since the classical local $\varepsilon$-factors depend on an additive character of the base field, a satisfied geometric $\varepsilon$-factorization theory will lie in an appropriate gerbe rather than be a super graded line (cf. \cite{Bei07, Pat12}).
\subsubsection{}More generally, we could also ask similar questions in a relative situation. Now let $f\colon X\rightarrow S$ be a proper morphism between smooth schemes over $k$. Let $\mathcal F$ be an $\ell$-adic sheaf on $X$ such that $f$ is universally locally acyclic relatively to $\mathcal F$.
Under these assumptions, we know that $Rf_\ast\mathcal F$ is locally constant on $S$. Now
we can ask if there is an analogue geometric $\varepsilon$-factorization for the determinant $\det Rf_\ast\mathcal F$. This problem is far beyond the authors' reach at this moment. But, similar to \eqref{eqYZ:ep}, we may consider  twist formulas for $\det Rf_\ast\mathcal F$.
One of the purposes of this paper is to formulate such a twist formula and prove it under  certain assumptions.
\subsubsection{Relative twist formula}
Let $S$ be a regular Noetherian scheme  over $\mathbb Z[1/\ell]$ and $f\colon X\to S$ a proper smooth morphism purely of relative dimension $n$.
Let $\mathcal F\in D_c^b(X,\Lambda)$ such that 
$f$ is universally locally acyclic relatively to $\mathcal F$.
Then we conjecture that (see Conjecture \ref{conj:rtf}) there exists a unique cycle class $cc_{X/S}(\mathcal F)\in{\rm CH}^n(X)$ such that for any locally constant and constructible sheaf $\mathcal G$ of $\Lambda$-modules on $X$, we have an isomorphism of smooth sheaves of rank 1 on $S$
\begin{align}\label{intro:rtf}
\det Rf_\ast(\mathcal F\otimes\mathcal G)\cong(\det Rf_\ast\mathcal F)^{\otimes\rank\mathcal G}\otimes \det\mathcal G(cc_{X/S}(\mathcal F) )
\end{align} 
where $\det\mathcal G(cc_{X/S}(\mathcal F)) $ is a smooth sheaf of rank 1 on $S$ (see \ref{subsub:defDetG} for the definition).
We call \eqref{intro:rtf} the relative twist formula. 
As an evidence, we prove a special case of the above conjecture in Theorem \ref{thm:rtf}.
It is also interesting to consider a similar relative twist formula for de Rham epsilon factors in the sense of \cite{AP17}. 
We will pursue this question elsewhere.

\subsubsection{}If $S$ is moreover a smooth connected scheme of dimension $r$ over a perfect field $k$, we construct a candidate for $cc_{X/S}(\mathcal F)$ in Definition \ref{def:rcclass}. We also relate the relative characteristic class $cc_{X/S}(\mathcal F)$ to the total characteristic class of $\mathcal F$. Let $K_0(X,\Lambda)$ be the Grothendieck group of $D^b_c(X,\Lambda)$.
In \cite[Definition 6.7.2]{Sai16}, T.~Saito defines the following morphism
\begin{align}
cc_{X,\bullet}\colon K_0(X,\Lambda)\to {\rm CH}_\bullet(X)=\bigoplus_{i= 0}^{r+n} {\rm CH}_i(X),
\end{align}
which sends $\mathcal F\in D_c^b(X,\Lambda)$ to the total characteristic class $cc_{X,\bullet}(\mathcal F)$ of $\mathcal F$.
Under the assumption  that $f\colon X\to S$ is $SS(\mathcal F, X/k)$-transversal,
we show that $(-1)^r\cdot cc_{X/S}(\mathcal F)=cc_{X, r}(\mathcal F)$ in Proposition \ref{prop:identificationoftwocc}.

\subsubsection{}
Following Grothendieck \cite{Gro77}, it's natural to ask whether the following diagram
\begin{align}\label{eq:fccfintro}
\begin{gathered}
\xymatrix{
K_0(X,\Lambda)\ar[d]_{f_\ast}\ar[r]^{cc_{X,\bullet}}&CH_\bullet(X)\ar[d]^{f_*}\\
K_0(Y,\Lambda)\ar[r]^{cc_{Y,\bullet}}&CH_\bullet(Y)
}
\end{gathered}
\end{align}
is commutative or not for any proper map $f\colon X\rightarrow Y$ between smooth schemes over $k$.
If $k=\mathbb C$, the diagram (\ref{eq:fccfintro}) is commutative by \cite[Theorem A.6]{Gin86}.
By the philosophy of Grothendieck, the answer is no in general if ${\rm char}(k)>0$ (cf. \cite[Example 6.10]{Sai16}).  
If $k$ is a finite field and if $f\colon X\to Y$ is moreover projective, as a corollary of Theorem \ref{thm:uyz}, we prove in \cite[Corollary 5.26]{UYZ} that the degree zero part of \eqref{eq:fccfintro} commutes. In general, motivated by the conjectural formula \eqref{intro:rtf}, we propose the following question.
Let $f\colon X\rightarrow S$ and $g\colon Y\to S$ be  smooth morphisms.
Let $D_{c}^b(X/S,\Lambda)$ be the thick subcategory of $D_c^b(X,\Lambda)$ consists of $\mathcal F\in D_c^b(X,\Lambda)$ such that $f$ is $SS(\mathcal F, X/k)$-transversal. 
Let $K_0(X/S,\Lambda)$ be the Grothendieck group of $D_{c}^b(X/S,\Lambda)$.
Then for any proper morphism $h\colon X\to Y$ over $S$,
we conjecture that the following diagram commutes (see Conjecture \ref{con:pushccr})
\begin{align}\label{eq:introPushccr}
\begin{gathered}
\xymatrix@C=3pc{
K_0(X/S,\Lambda)\ar[d]_{h_\ast}\ar[r]^-{~cc_{X,r}~}&CH_r(X)\ar[d]^{h_*}\\
K_0(Y/S,\Lambda)\ar[r]^-{~cc_{Y,r}~}&CH_r(Y).
}
\end{gathered}
\end{align}

\subsubsection{}As an evidence for \eqref{eq:introPushccr},
we  construct a relative cohomological characteristic  class 
\begin{align}ccc_{X/S}(\mathcal F)\in H^{2n}(X,\Lambda(n))
\end{align}
in Definition \ref{def:ccc} if $X\rightarrow S$ is smooth and $SS(\mathcal F, X/k)$-transversal. 
We prove that the formation of  $ccc_{X/S}(\mathcal F)$ is compatible with proper push-forward (see Corollary \ref{cor:pushccc} for a precise statement).
Similar to \cite[ Conjecture 6.8.1]{Sai16}, we conjecture that we have the following equality  (see Conjecture \ref{conj:ccequality})
\begin{equation}
{\rm cl}(cc_{X/S}(\mathcal F))=ccc_{X/S}(\mathcal F) \quad{\rm in}\quad  H^{2n}(X, \Lambda(n))
\end{equation}
where ${\rm cl}\colon {\rm CH}^n(X)\rightarrow  H^{2n}(X, \Lambda(n))$ is the cycle class map.

\subsection*{Acknowledgements}
Both authors are partially supported by the DFG through CRC 1085 \emph{Higher Invariants} (Universit\"at Regensburg). 
The authors are thank Naoya Umezaki for sharing ideas and for helpful  discussion during writing the paper \cite{UYZ}. 
The authors are also thank Professor Weizhe Zheng for discussing his paper \cite{Zh15}, and Haoyu Hu for his valuable comments.
The first author would like to thank his advisor Professor Linsheng Yin (1963-2015) for his constant encouragement during 2010-2015.

\subsection*{Notation and Conventions}

\begin{enumerate}
\item Let $p$ be a prime number and $\Lambda$ be a finite field of characteristic $\ell\neq p$ or $\Lambda=\overline{\mathbb Q}_\ell$.

\item We say that a complex $\mathcal F$ of \'etale sheaves of $\Lambda$-modules on a scheme $X$ over $\mathbb Z[1/\ell]$ is {\it constructible} (respectively {\it smooth}) if the cohomology sheaf $\mathcal H^q(\mathcal F)$ is constructible  for every $q$ and if $\mathcal H^q(\mathcal F)=0$ except finitely many $q$ (respectively moreover $\mathcal H^q(\mathcal F)$ is  locally constant for all $q$).

\item For a scheme $S$ over $\mathbb Z[1/{\ell}]$,
let $D^b_c(S,\Lambda)$ be the triangulated category of bounded complexes of $\Lambda$-modules with constructible cohomology groups on $S$
and let $K_0(S,\Lambda)$ be the Grothendieck group of $D^b_c(S,\Lambda)$. 

\item For a scheme $X$, we denote by  $|{X}|$  the set of closed points of  $X$.
\item For any smooth morphism $X\rightarrow S$, we denote by $T^\ast_X(X/S)\subseteq T^\ast (X/S)$ the zero section of the relative cotangent bundle $T^\ast (X/S)$ of $X$ over $S$. If $S$ is the spectrum of a field, we simply denote  $T^\ast (X/S)$ by $T^\ast X$. 
\end{enumerate}

\section{Relative twist formula}
\subsection{Reciprocity map}
\subsubsection{} For a smooth proper variety $X$ purely of dimension $n$ over a finite field $k$ of characteristic $p$, the reciprocity map  $ \rho_X\colon {\rm CH}^n(X) \to \pi^{\rm ab}_1(X)$ is given by sending the class $[s]$ of closed point $s\in X$ to the geometric Frobenius ${\rm Frob}_s$ at $s$. The map  $ \rho_X$ is injective with dense image \cite{KS83}.
\subsubsection{}
Let $S$ be a regular Noetherian scheme over $\mathbb Z[1/\ell]$ and $X$ a smooth proper scheme purely of relative dimension $n$ over $S$. By \cite[Proposition 1]{Sai94},
 there exists a unique way to attach a pairing 
\begin{align}\label{eq:1:100}
{\rm CH}^n(X)\times \pi_1^{\rm ab}(S)\to \pi_1^{\rm ab}(X) 
\end{align}
satisfying the following two conditions:
\begin{enumerate}
\item When $S={\rm Spec}k$ is a point, for a closed point $x\in X$, the pairing with the class $[x]$ coincides with the inseparable degree times the Galois transfer ${\rm tran}_{k(x)/k}$ (cf.\cite[1]{Tat79})  followed by $i_{x\ast}$ for $i_x\colon x\to X$
\[
{\rm Gal}(k^{\rm ab}/k)\xrightarrow{{\rm tran}_{k(x)/k}\times[k(x)\colon k]_i}
{\rm Gal}(k(x)^{\rm ab}/k(x))\xrightarrow{i_{x\ast}}\pi_1^{\rm ab}(X).
\]

\item For any point $s\in S$, the following diagram commutes
\[
\xymatrix{
{\rm CH}^n(X)\ar[d]\ar@{}|\times[r] &\pi_1^{\rm ab}(S)\ar[r]&\pi_1^{\rm ab}(X)\\
{\rm CH}^n(X_s)\ar@{}|\times[r]&\pi_1^{\rm ab}(s)\ar[u]\ar[r]&\pi_1^{\rm ab}(X_s).\ar[u]
}
\]

\end{enumerate}

\subsubsection{}\label{subsub:defDetG}
For any locally constant and constructible sheaf $\mathcal G$ of $\Lambda$-modules on $X$ and any $z\in {\rm CH}^n(X)$, we have a map
\begin{align}\label{eq:2:100}
\pi^{\rm ab}_1(S)\xrightarrow{(z,\bullet)}\pi_1^{\rm ab}(X)\xrightarrow{\det\mathcal G} \Lambda^\times
\end{align}
where $(z,\bullet)$ is the map determined by the paring \eqref{eq:1:100} and $\det\mathcal G$ is the representation associated to the locally constant sheaf
$\det\mathcal G$ of rank 1.
The composition $\det\mathcal G\circ (z,\bullet)\colon \pi^{\rm ab}_1(S)\to \Lambda^\times$ determines a locally constant and constructible sheaf of rank 1 on $S$, which we simply denote by $\det\mathcal G(z)$.
Now we propose the following conjecture.
\begin{conjecture}[Relative twist formula]\label{conj:rtf}
Let $S$ be a regular Noetherian scheme  over $\mathbb Z[1/\ell]$ and $f\colon X\to S$ a smooth proper morphism purely of relative dimension $n$.
Let $\mathcal F\in D_c^b(X,\Lambda)$ such that 
$f$ is universally locally acyclic relatively to $\mathcal F$.
Then there exists a unique cycle class $cc_{X/S}(\mathcal F)\in{\rm CH}^n(X)$ such that for any locally constant and constructible sheaf $\mathcal G$ of $\Lambda$-modules on $X$, we have an isomorphism
\begin{align}\label{eq:conjrtf100}
\det Rf_\ast(\mathcal F\otimes\mathcal G)\cong(\det Rf_\ast\mathcal F)^{\otimes\rank\mathcal G}\otimes \det\mathcal G(cc_{X/S}(\mathcal F)) \; \   \; {\rm in\ } K_0(S,\Lambda),
\end{align} 
where $K_0(S,\Lambda)$ is the Grothendieck group of $D_c^b(S,\Lambda)$.
\end{conjecture}
We call this cycle class $cc_{X/S}(\mathcal F)\in{\rm CH}^n(X)$ {\it the relative characteristic class of $\mathcal F$} if it exists. If $S$ is a smooth scheme over a perfect field $k$, we construct a candidate for $cc_{X/S}(\mathcal F)$ in Definition \ref{def:rcclass}.

As an evidence, we prove a special case of the above conjecture in Theorem \ref{thm:rtf}.
In order to construct a  cycle class $cc_{X/S}(\mathcal F)$ satisfying \eqref{eq:conjrtf100}, we use 
the theory of  singular support and  characteristic cycle.

\subsection{Transversal condition and singular support}\label{relcotbundle}
\subsubsection{}
Let $f\colon X\rightarrow S$ be a smooth morphism of Noetherian schemes  over $\mathbb Z[1/\ell]$. We denote by $T^\ast(X/S)$ the vector bundle ${\rm Spec}(\mathrm{Sym}_{\mathcal{O}_X}(\Omega^{1}_{X/S})^{\vee})$ on $X$ and call it {\it the relative cotangent bundle on} $X$ {\it with respect to} $S$. We denote by $T^*_X(X/S)=X$ the zero-section of
$T^\ast(X/S)$. A constructible subset $C$ of $T^\ast(X/S)$ is called {\it conical} if $C$ is invariant under the canonical $\mathbb G_m$-action on $T^\ast(X/S)$.

\begin{definition}[{\cite[\S 1.2]{Bei16} and \cite[\S 2]{HY17}}]
Let $f\colon X\rightarrow S$ be a smooth morphism of Noetherian schemes  over $\mathbb Z[1/\ell]$ and $C$ a closed conical subset of $T^\ast(X/S)$.
Let $Y$ be a Noetherian scheme smooth over $S$ and $h\colon Y\rightarrow X$ an $S$-morphism.  

(1) We say that  $h\colon Y\rightarrow X$ is {\it $C$-transversal relatively to} $S$ {\it at a geometric point} $\bar y\rightarrow Y$ if for every non-zero vector $\mu\in C_{h(\bar y)}=C\times_X\bar y$, the image $dh_{\bar y}(\mu)\in T_{\bar y}^\ast(Y/S)\coloneqq T^\ast(Y/S)\times_Y {\bar y}$ is not zero, where $dh_{\bar y}\colon T_{h(\bar y)}^\ast(X/S)\rightarrow T_{\bar y}^\ast(Y/S)$ is the canonical map. We say that $h\colon Y\rightarrow X$ is $C$-{\it transversal relatively to} $S$ if it is $C$-transversal relatively to $S$ at every geometric point of $Y$. If $h:Y\rightarrow X$ is $C$-transversal relatively to $S$, we put  $h^{\circ}C=dh(C\times_XY)$ where $dh:T^\ast(X/S)\times_X Y\rightarrow T^\ast(Y/S)$ is the canonical map induced by $h$.
By the same argument of \cite[Lemma 1.1]{Bei16}, $h^\circ C$ is a conical closed subset of $T^\ast(Y/S)$.

(2) Let $Z$ be a Noetherian scheme smooth over $S$ and $g\colon X\rightarrow Z$ an $S$-morphism. We say that $g\colon X\rightarrow Z$ is {\it $C$-transversal relatively to} $S$ {\it at a geometric point} $\bar x\rightarrow X$ if for every non-zero vector $\nu\in T_{g(\bar x)}^\ast(Z/S)$, we have $dg_{\bar x}(\nu)\notin C_{\bar x}$, where $dg_{\bar x}\colon T^*_{g(\bar x)}(Z/S)\rightarrow T^*_{\bar x}(X/S)$ is the canonical map. We say that $g\colon X\rightarrow Z$ is $C$-{\it transversal relatively to} $S$ if it is $C$-transversal relatively to $S$ at all geometric points of $X$.
If the base $B(C)\coloneq C\cap {T}_X^*(X/S) $ of $C$ is proper over $Z$, we put $g_\circ C:={\rm pr}_1(dg^{-1}(C))$, where ${\rm pr}_1\colon T^\ast (Z/S)\times_Z X\rightarrow T^\ast (Z/S)$ denotes the first projection and $dg: T^*(Z/S)\times_ZX\rightarrow T^*(X/S)$ is the canonical map. It is a closed conical subset of $T^*(Z/S)$.

(3) A {\it test pair of} $X$ {\it relative to} $S$ is a pair of $S$-morphisms $(g,h): Y\leftarrow U\rightarrow X$ such that $U$ and $Y$ are Noetherian schemes smooth over $S$. We say that $(g,h)$ is $C$-{\it transversal relatively to} $S$ if $h:U\rightarrow X$ is $C$-transversal relatively to $S$ and $g:U\rightarrow Y$ is $h^{\circ}C$-transversal relatively to $S$. 

\end{definition}

\begin{definition}[{\cite[\S 1.3]{Bei16} and \cite[\S 4]{HY17}}]
Let $f\colon X\rightarrow S$ be a smooth morphism of Noetherian schemes over $\mathbb Z[1/\ell]$. Let $\mathcal F$ be an object in $D^b_c(X,\Lambda)$.

(1) We say that a test pair $(g,h):Y\leftarrow U\rightarrow X$ relative to $S$ is $\mathcal F$-{\it acyclic} if $g:U\rightarrow Y$ is universally locally acyclic relatively to $h^*\mathcal F$.

(2) For a closed conical subset $C$ of ${T}^*(X/S)$, we say that $\mathcal F$ is {\it micro-supported on} $C$ {\it relatively to} $S$ if every $C$-transversal test pair of $X$ relative to $S$ is $\mathcal F$-acyclic. 

(3) Let  $\mathcal C(\mathcal F, X/S)$  be the set of all closed conical subsets $C'\subseteq T^*(X/S)$ such that $\mathcal F$ is 
micro-supported on $C'$ relatively to $S$.
Note that $\mathcal C(\mathcal F,X/S)$ is non-empty if $f:X\rightarrow S$ is universally locally acyclic relatively to $\mathcal F$.
If $\mathcal C(\mathcal F, X/S)$ has a smallest element, we denote it by $SS(\mathcal F, X/S)$ and call it the {\it singular support} of $\mathcal F$ {\it relative to} $S$. 
\end{definition}
\begin{theorem}[Beilinson]\label{existrss}
Let $f:X\rightarrow S$ be a smooth morphism between Noetherian schemes over $\mathbb{Z}[1/\ell]$ and  $\mathcal F$ an object of $D^b_c(X,\Lambda)$. 
\begin{itemize}
\item[(1)]$($\cite[Theorem 5.2]{HY17}$)$ If we further assume that $f:X\rightarrow S$ is projective and universally locally acyclic relatively to $\mathcal F$,  the singular support $SS(\mathcal F,X/S)$ exists.
\item[(2)]$($\cite[Theorem 5.2 and Theorem 5.3]{HY17}$)$ In general, after replacing $S$ by a Zariski open dense subscheme, the singular support $SS(\mathcal F,X/S)$ exists, and for any $s\in S$, we have
 \begin{equation}\label{ssequalwant}
   SS(\mathcal F|_{X_s}, X_s/s)=SS(\mathcal F,X/S)\times_Ss.
 \end{equation}
\item[(3)] $($\cite[Theorem 1.3]{Bei16}$)$ If $S={\rm Spec}k$ for a field $k$ and if $X$ is purely of dimension $d$, then $SS(\mathcal F, X/S)$ is purely of dimension $d$.
 \end{itemize}
\end{theorem}
\subsection{Characteristic cycle and index formula}
\subsubsection{}Let $k$ be a perfect field of characteristic $p$. 
Let $X$ be a smooth scheme purely of dimension $n$ over $k$, let $C$ be a closed conical subset of $T^*X$ and $f:X\rightarrow \mathbb A^1_k$ a $k$-morphism.
A closed point $v\in X$ is called {\it  at most an isolated $C$-characteristic point of $f:X\rightarrow \mathbb A^1_k$} if there is an open neighborhood $V\subseteq X$ of $v$ such that  $f: V-\{v\}\rightarrow \mathbb A^1_k$ is $C$-transversal. A closed point $v\in X$ is called an {\it isolated $C$-characteristic point} if $v$ is at most an isolated $C$-characteristic point of $f:X\rightarrow \mathbb A^1_k$ but  $f:X\rightarrow \mathbb A^1_k$ is not $C$-transversal at $v$.


\begin{theorem}[{T.~Saito, \cite[Theorem 5.9]{Sai16}}] 
Let $X$ be a smooth scheme purely of dimension $n$ over a perfect field $k$ of characteristic $p$.
Let $\mathcal F$ be an object of $D^b_c(X,\Lambda)$ and $\{C_\alpha\}_{\alpha\in I}$ the set of irreducible components of $SS(\mathcal F, X/k)$.  There exists a unique $n$-cycle $CC(\mathcal F, X/k)=\sum_{\alpha\in I}m_\alpha [C_\alpha]$ $(m_\alpha\in \mathbb Z)$ of $T^*X$ supported on $SS(\mathcal F, X/k)$,
satisfying the following Milnor formula (\ref{eq:milnor}):

For any \'etale morphism $g:V\rightarrow X$, any morphism $f:V\rightarrow\mathbb A^1_k$, any isolated $g^\circ SS(\mathcal F, X/k)$-characteristic point $v\in V$ of $f:V\rightarrow\mathbb A^1_k$ and any geometric point $\bar v$ of $V$ above $v$, we have
\begin{equation}\label{eq:milnor}
  -\dimtot ~{\rm R}\Phi_{\bar v}(g^\ast \mathcal F, f)=(g^*CC(\mathcal F,X/k),df)_{T^*V,v},
\end{equation}
where ${\rm R}\Phi_{\bar v}(g^*\mathcal F,f)$ denotes the stalk at $\bar v$ of the vanishing cycle complex of $g^*\mathcal F$ relative to $f$, $\dimtot~{\rm R}\Phi_{\bar v}(g^*\mathcal F,f)$ is the total dimension of  ${\rm R}\Phi_{\bar v}(g^*\mathcal F,f)$  and $g^*CC(\mathcal F,X/k)$ is the pull-back of $CC(\mathcal F,X/k)$ to $T^*V$. 
\end{theorem}
We call $CC(\mathcal F,X/k)$ the {\it characteristic cycle of} $\mathcal F$. It satisfies the following index formula.

\begin{theorem}[{T.~Saito, \cite[Theorem 7.13]{Sai16}}]\label{indexformula}
Let $\bar k$ be an algebraic closure of a perfect field $k$ of characteristic $p$, $X$ a smooth projective scheme over $k$ and $\mathcal F\in D^b_c(X,\Lambda)$. Then, we have
\begin{equation}\label{indexformulaeq}
\chi(X_{\bar k},\mathcal F|_{X_{\bar k}})=\deg(CC(\mathcal F,X/k),T^*_XX)_{T^*X},
\end{equation}
where $\chi(X_{\bar k},\mathcal F|_{X_{\bar k}})$ denotes the Euler-Poincar\'e characteristic of $\mathcal F|_{X_{\bar k}}$.
\end{theorem}
We give a generalization in Theorem \ref{thm:git}.  For a  smooth scheme $\pi\colon X\to {\rm Spec}k$, and two objects $\mathcal F_1$ and $\mathcal F_2$ in $D_c^b(X,\Lambda)$, we denote $\mathcal F_1\boxtimes^L_k\mathcal F_2\coloneq {\rm pr}_1^*\mathcal F_1 \otimes^L {\rm pr}_2^*\mathcal F_2 \in D^b_c(X\times X, \Lambda)$, where ${\rm pr}_i: X\times X \to X$ is the $i$th projection, for $i=1,2$. We also denote $D_X(\mathcal F_1)={R\mathcal Hom}(\mathcal F_1, \mathcal K_X)$, where $\mathcal K_X={\rm R}\pi^!\Lambda$.
\begin{lemma}\label{Lem:Kun}
Let $X$ be a smooth variety purely of dimension $n$ over a perfect field $k$ of characteristic $p$. Let $\mathcal F_1$ and $\mathcal F_2$ be two objects in $D_c^b(X,\Lambda)$. Then the diagonal map $\delta\colon \Delta=X\hookrightarrow X\times X$ is $SS(\mathcal F_2\boxtimes^LD_X\mathcal F_1,X\times X/k)$-transversal if and only if $SS(\mathcal F_2\boxtimes^L_k D_X\mathcal F_1,X\times X/k)\subseteq  {T}^\ast_{\Delta}(X\times X).$
 If we are in this case, then the canonical map
\[ R \mathcal Hom(\mathcal F_1, \Lambda)\otimes^L \mathcal F_2 \xrightarrow{\cong} R\mathcal Hom(\mathcal F_1, \mathcal F_2) \]
is an isomorphism.
\end{lemma}
\begin{proof}
The first assertion follows from the short exact sequence of vector bundles on $X$ associated to $\delta\colon \Delta=X\hookrightarrow X\times X$:
\[ 0 \to {T}^\ast_{\Delta}(X\times X) \to {T}^\ast(X\times X)\times_{X\times X}\Delta \xrightarrow{d\delta} T^*X \to 0 .   \]
For the second claim, we have the following canonical isomorphisms
\begin{align}
\nonumber R\mathcal Hom(\mathcal F_1, \Lambda)\otimes^L \mathcal F_2 &\cong R\mathcal Hom(\mathcal F_1, \Lambda(n)[2n])\otimes^L \Lambda(-n)[-2n]\otimes^L \mathcal F_2\overset{(1)}{\cong} D_X\mathcal F_1\otimes^L R\delta^!\Lambda\otimes^L \mathcal F_2\\
\label{isom} &\cong \delta^*(\mathcal F_2\boxtimes^L_kD_X\mathcal F_1)\otimes^L R\delta^!\Lambda \overset{(2)}{\cong} R\delta^!(\mathcal F_2\boxtimes^L_kD_X\mathcal F_1) \\
\nonumber &\overset{(3)}{\cong} R\delta^!(R\mathcal Hom({\rm pr}_2^*\mathcal F_1, R{\rm pr}_1^!\mathcal F_2))
\cong R\mathcal Hom(\delta^*{\rm pr}_2^*\mathcal F_1, R\delta^!R{\rm pr}_1^!\mathcal F_2)\\
\nonumber &\cong R\mathcal Hom( \mathcal F_1, \mathcal F_2),
\end{align}
where 

 (1) follows from the purity for the closed immersion $\delta$ \cite[XVI, Th\'eor\`eme 3.1.1]{ILO14};

 (2) follows from the assumption that $\delta$ is $SS(\mathcal F_2\boxtimes^L_kD_X\mathcal F_1)$-transversal by \cite[Proposition 8.13 and Definition 8.5]{Sai16};

 (3) follows from the K\"unneth formula \cite[Expos\'e III, (3.1.1)]{Gro77}.
\end{proof}
\begin{theorem}\label{thm:git}~
Let $X$ be a smooth projective variety purely of dimension $n$ over  an algebraically closed field $k$ of characteristic $p$.
Let $\mathcal F_1$ and $\mathcal F_2$ be two objects in $D_c^b(X,\Lambda)$ such that the diagonal map $\delta\colon \Delta=X\hookrightarrow X\times X$ is properly $SS(\mathcal F_2\boxtimes^L_kD_X\mathcal F_1,X\times X/k)$-transversal. Then we have
\begin{equation}\label{eq:git}
(-1)^n\cdot\dim_{\Lambda} {\rm Ext}(\mathcal F_1, \mathcal F_2)={\rm deg}\left(CC(\mathcal F_1,X/k), CC(\mathcal F_2,X/k) \right)_{{T}^\ast X}
\end{equation}
where $\dim_{\Lambda} {\rm Ext}(\mathcal F_1, \mathcal F_2)=\sum\limits_{i}(-1)^i\dim_{\Lambda} {\rm Ext}^i_{D_c^b(X,\Lambda)}(\mathcal F_1, \mathcal F_2)$.
\end{theorem}
\begin{proof}By the isomorphisms \eqref{isom}, the left hand side of (\ref{eq:git}) equals to 
\begin{align}
\nonumber (-1)^n\cdot\chi(X,R\delta^!(\mathcal F_2\boxtimes^L_kD_X\mathcal F_1))&=(-1)^n\cdot\chi(X,\delta^*(\mathcal F_2\boxtimes^L_kD_X\mathcal F_1)\\
\label{eq:git01} &=(-1)^n\cdot\deg(CC(\delta^\ast(\mathcal F_2\boxtimes^L_k D_X(\mathcal F_1),X/k),T^*_XX)_{{T}^*X}.
\end{align}
Since $\delta\colon X\rightarrow X\times X$ is properly $SS(\mathcal F_2\boxtimes^L_kD_X\mathcal F_1,X\times X/k)$-transversal, we have 
\begin{align}
\label{first.isom} CC(\delta^\ast(\mathcal F_2\boxtimes^L_kD_X\mathcal F_1),X/k)&=(-1)^n \delta^\ast CC(D(\mathcal F_2\boxtimes^L_kD_X\mathcal F_1,X\times X/k) )\\
\label{second.isom}&=(-1)^n \delta^\ast (CC(\mathcal F_2,X/k)\times CC(\mathcal F_1,X/k)).
\end{align}
where the equality \eqref{first.isom} follows from \cite[Theroem 7.6]{Sai16}, and \eqref{second.isom} follows from \cite[Theorem 2.2.2]{Sai17}.
Consider the  following commutative diagram
 \[ \xymatrix{
 {T}^*X\times {T}^*X \ar@{=}[r] & {T}^*(X\times X) & {{T}}^\ast(X\times X)\times_{X\times X}\Delta \ar[l]_-{\rm pr} \ar[r]^-{d\delta}& {T}^*X\\
 &{{T}}^\ast X \ar[r]^-{\cong}\ar[u] \ar[ul]^{\rm diag} & {{T}}^\ast_{\Delta}(X\times X)\ar[r] \ar[u]& X. \ar[u]\ar@{}|\Box[ul]
 }\]
We have $\delta^\ast (CC(\mathcal F_2,X/k)\times CC(\mathcal F_1,X/k))=d\delta_\ast {\rm pr}^!(CC(\mathcal F_2,X/k)\times CC(\mathcal F_1,X/k))$ and 
\begin{equation*}
{\rm deg}( \delta^\ast (CC(\mathcal F_2,X/k)\times CC(\mathcal F_1,X/k)), {{T}}^\ast_XX)_{{T}^\ast X}={\rm deg}\left(CC(\mathcal F_1,X/k), CC(\mathcal F_2,X/k) \right)_{{{T}}^\ast X}.
\end{equation*}
Then \eqref{eq:git} follows from the above formula and \eqref{eq:git01}.
\end{proof}

\begin{remark}
If $\mathcal F_1$ is the constant sheaf $\Lambda$, then Theorem \ref{thm:git} is the index formula \eqref{indexformulaeq}.  
Theorem \ref{thm:git} can be viewed as the $\ell$-adic version of the global index formula in the setting of $\mathcal D_X$-modules (cf. \cite[Theorem 11.4.1]{Gin86}).
\end{remark}

\subsection{Relative twist formula}
\subsubsection{}Let $S$ be a Noetherian scheme over $\mathbb{Z}[1/\ell]$, $f:X\rightarrow S$ a smooth morphism of finite type and $\mathcal F$ an object of $D^b_c(X,\Lambda)$. 
Assume that the relative singular support $SS(\mathcal F, X/S)$ exists.
A cycle $B=\sum_{i\in I} m_i[B_i]$ in ${T}^*(X/S)$ is called the {\it characteristic cycle of} $\mathcal F$ {\it  relative to} $S$ if each $B_i$ is a subset of $SS(\mathcal F, X/S)$, each $B_i\rightarrow S$ is open and equidimensional and if, for any algebraic geometric point $\bar s$ of $S$, we have 
\begin{equation}
B_{\bar s}=\sum_{i\in I} m_i[(B_i)_{\bar s}]=CC(\mathcal F|_{X_{\bar s}}, X_{\bar s}/\bar s).
\end{equation}
We denote by $CC(\mathcal F, X/S)$ the characteristic cycle of $\mathcal F$ on $X$ relative to $S$. Notice that relative characteristic cycles may not exist in general.

\begin{proposition}[{T.~Saito, \cite[Proposition 6.5]{HY17}}]\label{flattransversal}
Let $k$ be a perfect field of characteristic $p$. Let $S$ be a smooth connected scheme of dimension $r$ over $k$, $f:X\rightarrow S$ a smooth morphism of finite type and $\mathcal F$ an object of $D^b_c(X,\Lambda)$. 
Assume that $f: X\rightarrow S$ is $SS(\mathcal F,X/k)$-transversal and that each irreducible component of $SS(\mathcal F, X/k)$ is open and equidimensional over $S$. 
Then the relative singular support $SS(\mathcal F, X/S)$ and the relative characteristic cycle $CC(\mathcal F, X/S)$ exist, and we have 
\begin{align}
\label{eqs:rss}SS(\mathcal F,X/S)&=\theta(SS(\mathcal F, X/k)),\\
\label{eqs:rcc}CC(\mathcal F, X/S)&=(-1)^r\theta_*(CC(\mathcal F, X/k)),
\end{align}
where $\theta:{T}^*X\rightarrow {T}^*(X/S)$ denotes the projection induced by the canonical map $\Omega^1_{X/k}\rightarrow\Omega^1_{X/S}$. 
\end{proposition}

\begin{definition}\label{def:rcclass}
Let $k$ be a perfect field of characteristic $p$ and $S$ a smooth connected scheme of dimension $r$ over $k$. Let $f:X\rightarrow S$ be a smooth morphism purely of relative dimension $n$ and $\mathcal F$ an object of $D^b_c(X,\Lambda)$. 
Assume that $f: X\rightarrow S$ is $SS(\mathcal F,X/k)$-transversal.
Consider the following cartesian diagram
\begin{align}\label{eq:rcclass}
\xymatrix{
{T}^\ast S\times_SX\ar[r]\ar[d]&{T}^\ast X\ar[d]\\
X\ar[r]^-{0_{X/S}}&{T}^\ast(X/S)
}
\end{align}
where $0_{X/S}\colon X\rightarrow {T}^\ast(X/S)$ is the zero section.
Since $f: X\rightarrow S$ is $SS(\mathcal F,X/k)$-transversal, the refined Gysin pull-back $0_{X/S}^!(CC(\mathcal F,X/k))$ of $CC(\mathcal F,X/k)$ is a $r$-cycle class supported on $X$. We define the \emph{relative characteristic class} of $\mathcal F$ to be 
\begin{align}\label{eq:rcclassdef0}
cc_{X/S}(\mathcal F)=(-1)^r\cdot 0_{X/S}^!(CC(\mathcal F,X/k)) \quad{in}\quad {\rm CH}^n(X).
\end{align} 
\end{definition}

Now we prove a special case of Conjecture \ref{conj:rtf}.
\begin{theorem}[Relative twist formula]\label{thm:rtf}
Let $S$ be a smooth connected  scheme of dimension $r$ over a finite field $k$ of characteristic $p$.
 Let  $f\colon X\rightarrow S$ be a  smooth projective morphism of relative dimension $n$.
Let $\mathcal F\in D_c^b(X,\Lambda)$ and $\mathcal G$ a locally constant and constructible  sheaf of  $\Lambda$-modules on $X$.
Assume that $f$ is properly $SS(\mathcal F, X/k)$-transversal. Then there is an isomorphism 
\begin{align}
\det Rf_\ast(\mathcal F\otimes\mathcal G)\cong(\det Rf_\ast\mathcal F)^{\otimes\rank\mathcal G}\otimes \det\mathcal G(cc_{X/S}(\mathcal F)) \; \   \; {\rm in\ } K_0(S,\Lambda).
\end{align}
Note that we also have $cc_{X/S}(\mathcal F)=(CC(\mathcal F, X/S), {T}^*_XX)_{{T}^\ast(X/S)}\in {\rm CH}^n(X)$.
\end{theorem}
\begin{proof} We may assume $\mathcal G\neq 0$.
Since $\mathcal G$ is a smooth sheaf, we have $SS(\mathcal F, X/k)=SS(\mathcal F\otimes\mathcal G, X/k)$.
Since $f$ is proper and $SS(\mathcal F, X/k)$-transversal, by \cite[Lemma 4.3.4]{Sai16}, 
 $Rf_\ast \mathcal F$ and $Rf_\ast (\mathcal F\otimes\mathcal G)$ are smooth sheaves on $S$. For any closed point $s\in S$, we have the following commutative diagram 
 \[  \xymatrix{ {T}^*X\times_XX_s \ar[r]^-{\theta_{s}}\ar[d]^{\rm pr} &{T}^*X_s \cong {T}^*(X/S)\times_XX_s\ar[d]^{\rm pr} &X_s \ar[l]_-{0_{X_s}} \ar[d]^{i}\\
 \mathbb{T}^*X \ar[r]^{\theta} &{T}^*(X/S)\ar@{}|\Box[ul] & X\ar[l]_-{0_{X/S}}\ar@{}|\Box[ul]
  }\]
  where $0_{X/S}$ and $0_{X_s}$ are the zero sections. Hence we have 
  \begin{align}
 \nonumber cc_{X_s}(\mathcal F|_{X_s})&=(CC(\mathcal F|_{X_s},X_s/s), X_s)_{{T}^*X_s}=0_{X_s}^!CC(\mathcal F|_{X_s},X_s/s)\overset{(a)}{=}0_{X_s}^!i^!CC(\mathcal F, X/k)\\
 \nonumber &=(-1)^r0_{X_s}^!i^*CC(\mathcal F, X/k)=(-1)^r0_{X_s}^!\theta_{s*}{\rm pr}^!CC(\mathcal F, X/k)\\
 \label{eqs:rtf105}&=(-1)^r0_{X_s}^!{\rm pr}^!\theta_{*}CC(\mathcal F, X/k)
=(-1)^r0_{X_s}^!{\rm pr}^!((-1)^rCC(\mathcal F, X/S))\\
\nonumber&=0_{X_s}^!{\rm pr}^!CC(\mathcal F, X/S)=i^!0_{X/S}^!CC(\mathcal F, X/S)=i^!cc_{X/S}(\mathcal F),
  \end{align}
 where the equality (a) follows from \cite[Theorem 7.6]{Sai16} since $f$ is properly $SS(\mathcal F, X/k)$-transversal.
  
By Chebotarev density (cf. \cite[Th\'eor\`eme 1.1.2]{Lau87}), we may assume that $S$ is the spectrum of a finite field. Then it is sufficient to compare the Frobenius action.
Then one use \eqref{eqs:rtf105} and Theorem \ref{thm:uyz}.
\end{proof}

\begin{example}
Let $S$ be a smooth projective connected scheme over a finite field $k$ of characteristic $p>2$.
 Let  $f\colon X\rightarrow S$ be a  smooth projective morphism of relative dimension $n$, $\chi=\rank Rf_\ast\bQl$ the Euler-Poincar\'e number of the fibers
and let $\mathcal F$ be a constructible \'etale sheaf of $\Lambda$-modules on $S$. 
Then by the projection formula, we have $Rf_*f^*\mathcal F \cong \mathcal F\otimes Rf_*\overline{\mathbb Q}_{\ell}.$
Since $f$ is projective and smooth, $Rf_*\overline{\mathbb Q}_{\ell}$ is a smooth sheaf on $S$. Using Theorem \ref{thm:uyz}, we get 
\begin{align}
\varepsilon(S, Rf_\ast f^\ast\mathcal F)= \varepsilon(S, \mathcal F)^{\chi}\cdot {\rm det}Rf_*\overline{\mathbb Q}_{\ell}(-cc_{Y/k}(\mathcal F)).
\end{align}
By \cite[Theorem 2]{Sai94}, ${\rm det}Rf_*\overline{\mathbb Q}_{\ell}=\kappa_{X/S}(-\frac{1}{2}n\chi)$,
where $\kappa_{X/S}$ is a character of order at most 2 and is determined by the following way:

(1) If $n$ is odd, then $\kappa_{X/S}$ is trivial.

(2) If $n=2m$ is even, then $\kappa_{X/S}$ is the quadratic character defined by the square root of $(-1)^{\frac{\chi(\chi-1)}{2}}\cdot \delta_{{\rm dR},X/S}$, where
$\delta_{{\rm dR},X/S}\colon (\det H_{\rm dR}(X/S))^{\otimes 2}\xrightarrow{\simeq}\mathcal O_S$ is the de Rham discriminant defined by the  non-degenerate symmetric bilinear form $H_{\rm dR}(X/S)\otimes^{L}H_{\rm dR}(X/S)\to \mathcal O_S[-2n]$, and  $H_{\rm dR}(X/S)=Rf_*\Omega_{X/S}^{\bullet}$ is the perfect complex of $\mathcal O_S$-modules whose cohomology computes the relative de Rham cohomology of $X/S$.

Similarly, if $\mathcal F$ is a locally constant and constructible \'etale sheaf of $\Lambda$-modules on $S$, then 
\begin{align}
\nonumber{\rm det}Rf_\ast f^\ast\mathcal F\cong {\rm det}(\mathcal F\otimes Rf_*\overline{\mathbb Q}_{\ell})&\cong ({\rm det}\mathcal F)^{\otimes\chi}\otimes ({\rm det}Rf_*\overline{\mathbb Q}_{\ell} )^{\otimes{\rm rank}\mathcal F}\\
 &\cong ({\rm det}\mathcal F)^{\otimes\chi}\otimes (\kappa_{X/S}(-\frac{1}{2}n\chi))^{{\otimes\rm rank}\mathcal F}.
\end{align}
\end{example}

\subsection{Total characteristic class}
\subsubsection{}In the rest of this section, we relate the relative characteristic class $cc_{X/S}(\mathcal F)$ to the total characteristic class of $\mathcal F$.
Let $X$ be a smooth scheme purely of dimension $d$ over a perfect field $k$ of characteristic $p$.
In \cite[Definition 6.7.2]{Sai16}, T.~Saito defines the following morphism
\begin{align}\label{eq:totcc}
cc_{X,\bullet}\colon K_0(X,\Lambda)\to {\rm CH}_\bullet(X)=\bigoplus_{i= 0}^d {\rm CH}_i(X),
\end{align}
which sends $\mathcal F\in D_c^b(X,\Lambda)$ to the total characteristic class $cc_{X,\bullet}(\mathcal F)$ of $\mathcal F$.
For our convenience, for any integer $n$ we put
\begin{align}
cc_{X}^n(\mathcal F)\coloneq cc_{X,d-n}(\mathcal F)\quad{\rm in}\quad {\rm CH}^n(X).
\end{align}
By \cite[Lemma 6.9]{Sai16}, for any $\mathcal F\in D_c^b(X,\Lambda)$, we have 
\begin{align}
cc_{X}^d(\mathcal F)=cc_{X,0}(\mathcal F)&=(CC(\mathcal F, X/k), {T}^\ast_XX)_{{T}^\ast X} \quad{\rm in}\quad {\rm CH}_0(X),\\
cc_{X}^0(\mathcal F)= cc_{X,d}(\mathcal F)&=(-1)^d\cdot {\rm rank}\mathcal F\cdot [X]\,\, \quad\quad\quad{\rm in}\quad {\rm CH}_d(X)=\mathbb Z.
\end{align}
The following proposition gives a computation of $cc_{X}^n\mathcal F$ for any $n$.
\begin{proposition}\label{prop:identificationoftwocc}
Let $S$ be a smooth connected  scheme of dimension $r$ over a perfect  field $k$ of characteristic $p$.
Let $f\colon X\rightarrow S$ be a smooth morphism purely of relative dimension $n$.
Assume that $f$ is $SS(\mathcal F, X/k)$-transversal. Then we have
\begin{align}\label{eq:identificationoftwocc}
cc_{X}^n(\mathcal F)= (-1)^r\cdot  cc_{X/S}(\mathcal F) \quad{\rm in}\quad {\rm CH}^n(X)
\end{align}
where $cc_{X/S}(\mathcal F)$ is defined in Definition \ref{def:rcclass}.
\end{proposition}
\begin{proof}
We use the notation of \cite[Lemma 6.2]{Sai16}.
We put $F=({T}^\ast S\times_SX)\oplus \mathbb A_X^1$ and $E={T}^\ast X\oplus \mathbb A_X^1$. 
We have a canonical injection $i\colon F\to E$ of vector bundles on $X$ induced by
$df\colon T^\ast S\times_SX\rightarrow T^\ast X$.
Let $\bar i\colon \mathbb P(F)\rightarrow \mathbb P(E)$ be the canonical 
map induced by $i\colon F\to E$.
By \cite[Lemma 6.1.2 and Lemma 6.2.1]{Sai16}, we have a commutative diagram:
\begin{align}\label{eq:identificationoftwocc101}
\begin{gathered}
\xymatrix{
{\rm CH}_r(\mathbb P(F))&&{\rm CH}_{n+r}(\mathbb P(E))\ar[ll]_-{\bar i^\ast}\\
\bigoplus\limits_{q=0}^r{\rm CH}_q(X)\ar[u]^-{\simeq}\ar[d]_-{\rm can}&&\bigoplus\limits_{q=0}^{n+r}{\rm CH}_q(X)\ar[ll]_-{\rm can}\ar[u]^-{\simeq}\ar[d]_-{\rm can}\\
{\rm CH}_r(X)\ar@{=}[r]&{\rm CH}^n(X)\ar@{=}[r]&{\rm CH}_{r}(X).
}
\end{gathered}
\end{align}
Since $f$ is smooth and $SS(\mathcal F, X/k)$-transversal, the intersection
$SS(\mathcal F, X/k)\cap ({T}^\ast S\times_SX)$ is contained in the zero section of ${T}^\ast S\times_SX$. Thus the Gysin pull-back $i^\ast ({CC(\mathcal F, X/k)})$ is supported on the zero section of ${T}^\ast S\times_SX$. Let $\overline{CC(\mathcal F, X/k)}$ be any extension of $CC(\mathcal F, X/k)$ to $\mathbb P(E)$ (cf. \cite[Definition 6.7.2]{Sai16}).
Then $\bar i^\ast (\overline{CC(\mathcal F, X/k)})$ is an extension of $i^\ast ({CC(\mathcal F, X/k)})$ to $\mathbb P(F)$. 
By \cite[Definition 6.7.2]{Sai16}, the image of $\overline{CC(\mathcal F, X/k)}$ in ${\rm CH}^n(X)$ by the right vertical map of \eqref{eq:identificationoftwocc101} equals to $cc_{X}^n(\mathcal F)=cc_{X,r}(\mathcal F)$.
The image of $\bar i^\ast (\overline{CC(\mathcal F, X/k)})$ in ${\rm CH}^n(X)$ by the left vertical map of \eqref{eq:identificationoftwocc101} equals to $(-1)^r\cdot cc_{X/S}(\mathcal F)$ (cf. \eqref{eq:rcclassdef0}).
Now the equality \eqref{eq:identificationoftwocc} follows from the commutativity of \eqref{eq:identificationoftwocc101}.
\end{proof}

\subsubsection{}Following Grothendieck \cite{Gro77}, it's natural to ask the following question:
is the diagram
\begin{align}\label{eq:fccf}
\begin{gathered}
\xymatrix{
K_0(X,\Lambda)\ar[d]_{f_\ast}\ar[r]^{cc_{X,\bullet}}&CH_\bullet(X)\ar[d]^{f_*}\\
K_0(Y,\Lambda)\ar[r]^{cc_{Y,\bullet}}&CH_\bullet(Y)
}
\end{gathered}
\end{align}
commutative for any proper map $f\colon X\rightarrow Y$ between smooth schemes over a perfect field $k$? 
If $k=\mathbb C$, the diagram (\ref{eq:fccf}) is commutative by \cite[Theorem A.6]{Gin86}.
By the philosophy of Grothendieck, the answer is no in general if ${\rm char}(k)>0$ (cf. \cite[Example 6.10]{Sai16}).  
However, in \cite[Corollary 1.9]{UYZ}, we prove that the degree zero part of the diagram (\ref{eq:fccf}) is commutative, i.e., 
if $f\colon X\rightarrow Y$ is a proper map between smooth projective schemes over a finite field $k$ of characteristic $p$, then we have  the following commutative diagram
\begin{align}\label{eq:fccf2}
\begin{gathered}
\xymatrix{
K_0(X,\Lambda)\ar[d]_{f_\ast}\ar[r]^{cc_{X,0}}&CH_0(X)\ar[d]^{f_*}\\
K_0(Y,\Lambda)\ar[r]^{cc_{Y,0}}&CH_0(Y).
}
\end{gathered}
\end{align}
Now we propose the following:
\begin{conjecture}\label{con:pushccr}
Let $S$ be a smooth connected  scheme over a perfect  field $k$ of characteristic $p$.
Let $f\colon X\rightarrow S$ be a smooth morphism purely of relative dimension $n$ and  $g\colon Y\to S$ a smooth morphism purely of relative dimension $m$.
Let $D_{c}^b(X/S,\Lambda)$ be the thick subcategory of $D_c^b(X,\Lambda)$ consists of $\mathcal F\in D_c^b(X,\Lambda)$ such that $f\colon X\rightarrow S$ is $SS(\mathcal F, X/k)$-transversal. 
Let $K_0(X/S,\Lambda)$ be the Grothendieck group of $D_{c}^b(X/S,\Lambda)$.
Then for any proper morphism $h\colon X\to Y$ over $S$,
\begin{align}\label{eq:con:pushccr1}
\begin{gathered}
\xymatrix{
X\ar[rr]^-h\ar[rd]_-f&&Y\ar[ld]^-g\\
&S
}
\end{gathered}
\end{align}
 the following diagram commutes
\begin{align}\label{eq:con:pushccr2}
\begin{gathered}
\xymatrix@C=3pc{
K_0(X/S,\Lambda)\ar[d]_{h_\ast}\ar[r]^-{~cc_{X}^n~}&CH^n(X)\ar[d]^{h_*}\\
K_0(Y/S,\Lambda)\ar[r]^-{~cc_{Y}^m~}&CH^m(Y).
}
\end{gathered}
\end{align}
That is to say, for any $\mathcal F\in D_c^b(X,\Lambda)$, if $f$ is $SS(\mathcal F, X/k)$-transversal, then we have
\begin{align}
h_\ast(cc_X^n(\mathcal F))=cc_Y^m(Rh_\ast\mathcal F)\quad{\rm in}\quad CH^m(Y).
\end{align}
\end{conjecture}
\begin{remark}
If $f$ is $SS(\mathcal F, X/k)$-transversal, 
by \cite[Lemma 3.8 and Lemma 4.2.6]{Sai16}, the morphism $g\colon Y\to S$ is $SS(Rh_\ast\mathcal F, Y/k)$-transversal.
Thus we have a well-defined map $h_\ast\colon K_0(X/S,\Lambda)\to K_0(Y/S,\Lambda)$.
\end{remark}
In next section, we formulate and prove a cohomological version of Conjecture \ref{con:pushccr} (cf. Corollary \ref{cor:pushccc}).

\section{Relative cohomological characteristic class}\label{sec:rccc}
In this section, we assume that $S$ is a smooth connected scheme over a perfect field  $k$ of characteristic $p$ and $\Lambda$ is a finite field of characteristic $\ell$.  To simplify our notations, we omit to write $R$ or $L$ to denote the derived functors unless otherwise stated explicitly or for $R\mathcal Hom$.

We briefly recall the content of this section.
Let $X\to S$ be a smooth morphism   purely of relative dimension $n$ and $\mathcal F\in D_c^b(X,\Lambda)$.
If $X\to S$ is $SS(\mathcal F, X/k)$-transversal,
we construct a relative cohomological characteristic class $ccc_{X/S}(\mathcal F)\in H^{2n}(X,\Lambda(n))$ following the method of \cite{AS07,Gro77}. 
We  conjecture that the image of the 
cycle class $cc_{X/S}(\mathcal F)$ by the cycle class map 
${\rm cl}: {\rm CH}^n(X)\rightarrow H^{2n}(X,\Lambda(n))$ is $ccc_{X/S}(\mathcal F)$ (cf. Conjecture \ref{conj:rtf}).
In Corollary \ref{cor:pushccc}, we prove that the formation of  $ccc_{X/S}\mathcal F$ is compatible with proper push-forward.
\subsection{Relative cohomological correspondence}
\subsubsection{}\label{Sec4:notation} Let $\pi_1\colon X_1\to S$ and $ \pi_2\colon X_2\to S$ be smooth morphisms purely of relative dimension $n_1$ and $n_2$ respectively. We put $X\coloneq X_1\times_SX_2 $ and  consider the following cartesian diagram 
\begin{align}
\begin{gathered}
\xymatrix{ X\ar[r]^-{\rm pr_2} \ar[d]_-{\rm pr_1}& X_2 \ar[d]^-{\pi_2}\\ X_1 \ar[r]_-{\pi_1} &S. \ar@{}|\Box[ul]     
}
\end{gathered}
\end{align}
Let $\mathcal E_i$ and $\mathcal F_i$ be  objects of $D^b_c(X_i,\Lambda)$ for $i=1,2$.
We put 
\begin{align}
&\mathcal F\coloneq\mathcal F_1\boxtimes^L_S\mathcal F_2\coloneqq {\rm pr}_1^\ast \mathcal F_1 \otimes^L {\rm pr}_2^\ast \mathcal F_2,\\
&\mathcal E\coloneq\mathcal E_1\boxtimes^L_S\mathcal E_2\coloneqq {\rm pr}_1^\ast \mathcal E_1 \otimes^L {\rm pr}_2^\ast \mathcal E_2,
\end{align}
which are objects of $D_c^b(X,\Lambda)$.
Similarly, we can define $\mathcal F_1\boxtimes^L_k\mathcal F_2$, which is an object of $D_c^b(X_1\times_k X_2,\Lambda)$.
We first compare $SS(\mathcal F_1\boxtimes_S^L\mathcal F_2, X_1\times_S X_1/k)$ and $SS(\mathcal F_1\boxtimes_k^L\mathcal F_2, X_1\times_k X_1/k)$. 
\begin{lemma}\label{lem:boxtimesTrans}
Assume that $\pi_1\colon X_1\to S$ is $SS(\mathcal F_1, X_1/k)$-transversal. Then we have
\begin{align}\label{eq:lem:boxtimesTrans00}
SS({\rm pr}_1^\ast \mathcal F_1, X/k)\cap SS({\rm pr}_2^\ast \mathcal F_2, X/k)\subseteq T^\ast_XX.
\end{align}
Moreover, the closed immersion  $i\colon X_1\times_SX_2 \hookrightarrow X_1\times_kX_2$ is $SS(\mathcal F_1\boxtimes^L_k\mathcal F_2, X_1\times_kX_2/k)$-transversal and
\begin{equation}\label{eq:lem:boxtimesTrans}
SS(\mathcal F_1\boxtimes_S^L\mathcal F_2, X_1\times_SX_2/k)\subseteq i^{\circ}(SS(\mathcal F_1\boxtimes^L_k\mathcal F_2, X_1\times_k X_2/k)).
\end{equation}

\end{lemma}
\begin{proof}
We first prove \eqref{eq:lem:boxtimesTrans00}.
Since $X_i\to S$ is smooth, we obtain  an exact sequence of vector bundles on $X_i$ for $i=1,2$
\begin{align} \label{eq:lem:boxtimesTrans101}
0 \to {T}^*S\times_SX_i \xrightarrow{d\pi_i} {T}^*X_i\to {T}^*(X_i/S)\to 0.
\end{align}
Since $\pi_1\colon X_1\to S$ is $SS(\mathcal F_1, X_1/k)$-transversal, we have 
\begin{align}\label{eq:lem:boxtimesTrans102}
SS(\mathcal F_1,X_1/k)\cap (T^*S\times_SX_1)\subseteq T^*_SS\times_SX_1.
\end{align}
Consider the following diagram with exact rows and exact columns:
\begin{align}\label{eq:lem:boxtimesTrans103}
\begin{gathered}
\xymatrix@R=1pc@C=1.5pc{ &0&0&&\\
 & {T}^*(X_2/S)\times_{X_2}X \ar[r]^-{\cong} \ar[u]&{T}^*(X/X_1)  \ar[u] &&\\
0\ar[r] & {T}^*X_2\times_{X_2}X \ar[u]\ar[r] & {T}^*X\ar[u] \ar[r] & {T}^*(X/X_2) \ar[r] &0\\
 0\ar[r] & {T}^*S\times_{S}X \ar[u]\ar[r] & {T}^*X_{1}\times_{X_1}X \ar[r]\ar[u] & {T}^*(X_1/S)\times_{X_1}X \ar[r] \ar[u]_-{\cong} &0\\
&0\ar[u]&0 \ar[u]&& }
\end{gathered}
\end{align}
Since ${\rm pr}_i$ is smooth,
by \cite[Corollary 8.15]{Sai16}, we have 
\[
SS({\rm pr}_i^\ast \mathcal F_i, X/k)={\rm pr}_i^{\circ}SS(\mathcal F_i,X_i/k)=SS(\mathcal F_i,X_i/k)\times_{X_i}X.
\] 
It follows from \eqref{eq:lem:boxtimesTrans102} and \eqref{eq:lem:boxtimesTrans103} that 
${\rm pr}_1^{\circ}SS(\mathcal F_1,X_1/k)\cap {\rm pr}_2^{\circ}SS(\mathcal F_2,X_2/k) \subseteq  {T}_X^*X$.  
Thus $SS({\rm pr}_1^\ast \mathcal F_1, X/k)\cap SS({\rm pr}_2^\ast \mathcal F_2, X/k)\subseteq T^\ast_XX
$.
This proves  \eqref{eq:lem:boxtimesTrans00}.

Now we consider the cartesian diagram
\begin{align}\label{eq:lem:boxtimesTrans104}
\begin{gathered}
 \xymatrix{ X=X_1\times_SX_2 \ar[r]^-i \ar[d] & X_1\times_kX_2\ar[d]\\
S\ar[r]^-{\delta} &S\times_kS\ar@{}|\Box[ul]
}
\end{gathered}
\end{align}
where $\delta\colon S\rightarrow S\times_k S$ is the diagonal.
We get the following commutative diagram of vector bundles on $X$ with exact rows:
\[ 
\xymatrix@C=1pc@R=1pc{
&&T^\ast X_1\times_S T^\ast X_2\ar@{=}[d]\\
0\ar[r] &\mathcal N_{{X}/(X_1\times_kX_2)} \ar[r] & T^*(X_1\times_kX_2)\times_{X_1\times_kX_2}X \ar[r]^-{di} &T^*X\ar[r]&0\\
0\ar[r]&\mathcal N_{S/(S\times_kS)}\times_SX\ar[u]_-{\cong} \ar[r]&T^{*}(S\times_kS)\times_{S\times_kS}X\ar[u]\ar[r]^-{d\delta} &T^*S\times_S X\ar[u] \ar[r]&0\\
&T^\ast S\times_SX\ar@{=}[u]\ar[r]&(T^\ast S\times_S X_1)\times_S (T^\ast S\times_SX_2)\ar@{=}[u],
}\]
where $\mathcal N_{S/(S\times_kS)}$ is the conormal bundle associated to $\delta\colon S\to S\times_k S$.
By \cite[Theorem 2.2.3]{Sai17}, we have  $SS(\mathcal F_1\boxtimes_k^L\mathcal F_2, X_1\times_kX_2/k)=SS(\mathcal F_1, X_1/k)\times SS(\mathcal F_2, X_2/k)$. Therefore  by \eqref{eq:lem:boxtimesTrans102},
$\mathcal N_{{X}/(X_1\times_kX_2)} \cap SS(\mathcal F_1\boxtimes_k^L\mathcal F_2, X_1\times_kX_2/k)$ is 
contained in the zero section of $\mathcal N_{{X}/(X_1\times_kX_2)}$. 
Hence $i\colon X \hookrightarrow X_1\times_kX_2$ is 
$SS(\mathcal F_1\boxtimes^L_k\mathcal F_2, X_1\times_kX_2/k)$-transversal. 
Now the assertion \eqref{eq:lem:boxtimesTrans} follows from \cite[Lemma 4.2.4]{Sai16}.
\end{proof}

\begin{proposition}\label{pro:doubleRHOMiso} 
Under the notation in \ref{Sec4:notation},
we assume that
\begin{enumerate}
\item $SS(\mathcal E_i,X_i/k)\cap SS(\mathcal F_i,X_i/k)\subseteq T^*_{X_i}X_i $ for all $i=1,2$;
\item $\pi_1\colon X_1\to S$ is $SS(\mathcal E_1,X_1/k)$-transversal or $\pi_2\colon X_2\to S$ is $SS(\mathcal F_2,X_2/k)$-transversal;
\item  $\pi_1\colon X_1\to S$ is $SS(\mathcal F_1,X_1/k)$-transversal or $\pi_2\colon X_2\to S$ is $SS(\mathcal E_2,X_2/k)$-transversal.
\end{enumerate}
Then the following canonical map $($cf. \cite[(7.2.2)]{Zh15} and \cite[Expos\'e III, (2.2.4)]{Gro77}$)$
\begin{align}
\label{tensor.pairing} R\mathcal Hom(\mathcal E_1,\mathcal F_1)\boxtimes^L_S R\mathcal Hom(\mathcal E_2,\mathcal F_2) \to R\mathcal Hom(\mathcal E,\mathcal F).
\end{align}
is an isomorphism.
\end{proposition}
If $S$ is the spectrum of a field, then the above result is proved in \cite[Expos\'e III, Proposition 2.3]{Gro77}. Our proof below is different from  that of \emph{loc.cit.} and is based on \cite{Sai16}.
\begin{proof}
In the following, we put $\mathcal E_i^{\vee}\coloneq R\mathcal Hom(\mathcal E_i, \Lambda) $.
Since $SS(\mathcal E_i,X_i/k)\cap SS(\mathcal F_i,X_i/k)\subseteq T^*_{X_i}X_i $, Lemma \ref{Lem:Kun} implies that 
\begin{equation}
\mathcal F_i\otimes^L \mathcal E_i^{\vee}= \mathcal F_i\otimes^LR\mathcal Hom(\mathcal E_i, \Lambda) \xrightarrow{\cong} R\mathcal Hom(\mathcal E_i, \mathcal F_i),\; {\rm for\; all }\;i=1,2,
\end{equation}
Hence we have
\begin{align}
\label{isom.lhs} R\mathcal Hom(\mathcal E_1,\mathcal F_1)\boxtimes^L_S R\mathcal Hom(\mathcal E_2,\mathcal F_2)\cong (\mathcal F_1\otimes^L\mathcal E_1^{\vee})\boxtimes^L_S(\mathcal F_2 \otimes^L\mathcal E_2^{\vee})\\
\nonumber \cong (\mathcal F_1\boxtimes^L_S\mathcal F_2)\otimes^L(\mathcal E_1^{\vee} \boxtimes^L_S\mathcal E_2^{\vee}).
\end{align}
Note that we also have
\begin{align}
\nonumber\mathcal E_1^{\vee} \boxtimes^L_S\mathcal E_2^{\vee}&={\rm pr}_1^*R\mathcal Hom(\mathcal E_1,\Lambda) \otimes^L{\rm pr}_2^*R\mathcal Hom(\mathcal E_2,\Lambda)\\
  \label{isom.dual} &\cong R\mathcal Hom({\rm pr}_1^*\mathcal E_1,\Lambda) \otimes^LR\mathcal Hom({\rm pr}_2^*\mathcal E_2,\Lambda)\\
 \nonumber &\overset{(a)}{\cong} R\mathcal Hom({\rm pr}_1^*\mathcal E_1, R\mathcal Hom({\rm pr}_2^*\mathcal E_2,\Lambda))\\
 \nonumber &\cong R\mathcal Hom({\rm pr}_1^*\mathcal E_1\otimes^L{\rm pr}_2^*\mathcal E_2,\Lambda)=\mathcal E^{\vee},
 \end{align}
 where the isomorphism (a) follows from Lemma \ref{Lem:Kun} by the fact that (cf. Lemma \ref{lem:boxtimesTrans})
\[SS({\rm pr}_1^*\mathcal E_1,X/k)\cap SS({\rm pr}_2^*\mathcal E_2,X/k)\subseteq T^*_XX.
\]
By Lemma \ref{lem:boxtimesTrans}, we have
\begin{align}
\nonumber&SS(\mathcal E,X/k)\cap SS(\mathcal F,X/k)\\
\nonumber&\subseteq i^\circ(SS(\mathcal E_1\boxtimes^L_k\mathcal E_2, X_1\times_k X_2/k))\cap i^\circ(SS(\mathcal F_1\boxtimes^L_k\mathcal F_2, X_1\times_k X_2/k))\\
\nonumber&\overset{(b)}{=}i^\circ(SS(\mathcal E_1, X_1)\times SS(\mathcal E_2, X_2))\cap i^\circ(SS(\mathcal F_1, X_1)\times SS(\mathcal F_2, X_2))\\
\nonumber&\overset{(c)}{\subseteq} T^\ast _XX,
\end{align}
where the equality (b) follows from \cite[Theorem 2.2.3]{Sai17}, and $(c)$ follows from the assumptions (2) and (3) (cf. \cite[Lemma 2.7.2]{Sai17}). Thus by Lemma \ref{Lem:Kun}, we have
\begin{align}\label{eq:isomKunDual}
\mathcal F\otimes^L \mathcal E^{\vee}\cong R\mathcal Hom(\mathcal E, \mathcal F).
\end{align}
Combining \eqref{isom.lhs}, \eqref{isom.dual} and \eqref{eq:isomKunDual}, we get
\begin{align}
R\mathcal Hom(\mathcal E_1,\mathcal F_1)\boxtimes^L_S R\mathcal Hom(\mathcal E_2,\mathcal F_2)\cong \mathcal F\otimes^L \mathcal E^{\vee}\cong R\mathcal Hom(\mathcal E, \mathcal F).
\end{align}
This finishes the proof.
\end{proof}

\subsubsection{K\"unneth formula}
We have the following canonical morphism
\begin{equation}\label{eqs:Kunn}
\mathcal F_1\boxtimes^L_S R\mathcal Hom(\mathcal F_2, \pi_2^!\Lambda_S) \to R\mathcal Hom({\rm pr}_2^*\mathcal F_2, {\rm pr}_1^!\mathcal F_1),
\end{equation}
by taking the adjunction of the following composition map
\begin{align}
\nonumber {\rm pr}_1^*\mathcal F_1\otimes {\rm pr}_2^*R\mathcal Hom(\mathcal F_2, \pi_2^!\Lambda_S) &\otimes {\rm pr}_2^*\mathcal F_2 \to {\rm pr}_1^*\mathcal F_1\otimes {\rm pr}_2^*(\mathcal F_2 \otimes R\mathcal Hom(\mathcal F_2, \pi_2^!\Lambda_S))\\
\nonumber &\xrightarrow{\rm evaluation}  {\rm pr}_1^*\mathcal F_1\otimes {\rm pr}_2^*\pi_2^!\Lambda_S 
\to {\rm pr}_1^*\mathcal F_1\otimes{\rm pr}_1^!\Lambda_{X_1} \to {\rm pr}_1^!\mathcal F_1.
\end{align}

\begin{corollary}\label{cor:SSKunn}
Assume that  $\pi_1\colon X_1\to S$ is $SS(\mathcal F_1, X_1/k)$-transversal
or $\pi_2\colon X_2\to S$ is $SS(\mathcal F_2, X_2/k)$-transversal.
Then the canonical map \eqref{eqs:Kunn} is an isomorphism.
\end{corollary}
If $S$ is the spectrum of a  field, then the above result is proved in \cite[Expos\'e III, (3.1.1)]{Gro77}. Our proof below is different from that of \emph{loc.cit}.

\begin{proof}
By Proposition \ref{pro:doubleRHOMiso}, we have the following isomorphisms 
\begin{align}
\nonumber \mathcal F_1\boxtimes^L_S R\mathcal Hom(\mathcal F_2, \pi_2^!\Lambda_S) &\overset{Prop.{\ref{pro:doubleRHOMiso}}}{\cong} R\mathcal Hom({\rm pr}_2^*\mathcal F_2, {\rm pr}_1^\ast\mathcal F_1\otimes {\rm pr}_1^!\Lambda_S)\\
\nonumber&\overset{(a)}{\cong} R\mathcal Hom({\rm pr}_2^*\mathcal F_2, {\rm pr}_1^!\mathcal F_1),
\end{align}
where $(a)$ follows from the fact that ${\rm pr}_1$ is smooth (cf. \cite[XVI, Th\'eor\`eme 3.1.1]{ILO14} and \cite[XVIII, Theor\'eme 3.2.5]{SGA4T3}).
\end{proof}
\begin{definition}
Let  $X_i, \mathcal F_i$  be as in \ref{Sec4:notation} for $i=1,2$. A \emph{relative correspondence} between $X_1$ and $X_2$ is a scheme $C$ over $S$ with morphisms $c_1\colon C\to X_1$ and $c_2\colon C\to X_2$ over $S$. We put $c=(c_1,c_2)\colon C\to X_1\times_S X_2$ the corresponding morphism. A morphism $u\colon c_2^*\mathcal F_2\to c_1^!\mathcal{F}_1$ is called a \emph{relative cohomological correspondence} from $\mathcal F_2$ to $\mathcal{F}_1$ on $C$.
\end{definition}
\subsubsection{}Given a correspondence $C$ as above, we recall that there is a canonical isomorphism \cite[XVIII, 3.1.12.2]{SGA4T3}
\begin{equation}\label{eq:pullbackupp}
R\mathcal Hom(c_2^*\mathcal F_2, c_1^!\mathcal F_1)
\xrightarrow{\cong}
c^!R\mathcal Hom({\rm pr}_2^*\mathcal F_2, {\rm pr}_1^!\mathcal F_1).
\end{equation}

\subsubsection{}
For $i=1,2$, consider the following diagram of $S$-morphisms
\[ \xymatrix{ X_i \ar[rr]^{f_i}\ar[dr]_{\pi_i} && Y_i \ar[dl]^{q_i}\\
&S,&
}\]
where $\pi_i$ and $q_i$ are smooth morphisms. We put $X\coloneq X_1\times_SX_2$, $Y\coloneq Y_1\times_SY_2$ and $f\coloneq f_1\times_Sf_2\colon X \to  Y$. 
Let $\mathcal M_i\in D^b_c(Y_i,\Lambda)$ for $i=1,2$. We  have a canonical map (cf. \cite[Construction 7.4]{Zh15} and \cite[Expos\'e III, (1.7.3)]{Gro77})
\begin{equation}\label{morp.Kunneth}
f_1^!\mathcal M_1\boxtimes^L_S f_2^!\mathcal M_2 \to f^!(\mathcal M_1\boxtimes^L_S\mathcal M_2)
\end{equation}
which is adjoint to the composite 
\begin{align}
f_!(f_1^!\mathcal M_1\boxtimes^L_S f_2^!\mathcal M_2)\xrightarrow[(a)]{\simeq} f_{1!}f_1^!\mathcal M_1\boxtimes^L_S f_{2!}f_2^!\mathcal M_2\xrightarrow{{\rm adj}\boxtimes{\rm adj}}\mathcal M_1\boxtimes^L_S\mathcal M_2
\end{align}
where (a) is the K\"unneth isomorphism \cite[XVII, Th\'eor\`eme 5.4.3]{SGA4T3}. 
\begin{proposition}\label{prop:fUpperIsomKun}
If $q_2\colon Y_2\to S$ is $SS(\mathcal M_2, Y_2/k)$-transversal, then the map \eqref{morp.Kunneth} is an isomorphism.
\end{proposition}
If $S$ is the spectrum of a field,  the above result is proved in \cite[Expos\'e III, Proposition 1.7.4]{Gro77}.
\begin{proof}
Consider the following cartesian  diagrams
\[ \xymatrix{
X_1\times_SX_2 \ar[r]^{f_1\times {\rm id}}\ar[dr]^f\ar[d]_{{\rm id}\times f_2} &Y_1\times_SX_2\ar[r] \ar[d]^{{\rm id}\times f_2} &X_2\ar[d]^{f_2}\\
X_1\times_SY_2 \ar[d]_{\rm pr_1}\ar[r]^{f_1\times {\rm id}}&Y_1\times_SY_2\ar[d]^{\rm pr_1} \ar[r]^-{\rm pr_2} & Y_2\ar[d]^{q_2}\\
X_1\ar[r]^{f_1}\ar[dr]_{\pi_1} &Y_1 \ar[r]^{q_1} \ar[d]^{q_1} & S\\
&S.&
}\]
We may assume that $X_2=Y_2$ and $f_2={\rm id}$, i.e., it suffices to show that the canonical map
\begin{equation}\label{eq:prop:fUpperIsomKun100}
f_1^!\mathcal M_1\boxtimes^L_S\mathcal M_2 \xrightarrow{\cong} (f_1\times {\rm id})^!(\mathcal M_1\boxtimes^L_S \mathcal M_2).
\end{equation}
is an isomorphism.
Since we have
\begin{align}
\nonumber\mathcal M_2\cong D_{Y_2}D_{Y_2}\mathcal M_2&\cong R\mathcal Hom(D_{Y_2}\mathcal M_2, \mathcal K_{Y_2})\\
\nonumber&\cong R\mathcal Hom(D_{Y_2}(\mathcal M_2)(-{\rm dim}S)[-2{\rm dim}S], q_2^!\Lambda_S),
\end{align}
we may assume $\mathcal M_2= R\mathcal Hom(\mathcal L_2, q_2^!\Lambda_S)$ for some $\mathcal L_2\in D^b_c(Y_2, \Lambda)$.
By \cite[Corollary 4.9]{Sai16}, we have $SS(\mathcal M_2, Y_2/k)=SS(\mathcal L_2, Y_2/k)$.
Thus by assumption, the morphism $q_2\colon Y_2\to S$ is $SS(\mathcal L_2, Y_2/k)$-transversal.
By Corollary \ref{cor:SSKunn},  we have an isomorphism
\begin{align}\label{eq:prop:fUpperIsomKun101}
&\mathcal M_1\boxtimes^L_S R\mathcal Hom(\mathcal L_2, q_2^!\Lambda_S)\cong R\mathcal Hom({\rm pr}_2^\ast\mathcal L_2, {\rm pr}_1^!\mathcal M_1) \quad{\rm in}\, D_c^b(Y_1\times_S Y_2,\Lambda),\\
\label{eq:prop:fUpperIsomKun102}&
f_1^!\mathcal M_1\boxtimes^L_S R\mathcal Hom (\mathcal L_2, q_2^!\Lambda_S)\cong
R\mathcal Hom((f_1\times {\rm id})^*{\rm pr}_2^*\mathcal L_2, {\rm pr}_1^!f_1^!\mathcal M_1)  \,\,{\rm in}\, D_c^b(X_1\times_S Y_2,\Lambda). 
\end{align}
We have 
\begin{align}
\nonumber (f_1\times {\rm id})^!(\mathcal M_1\boxtimes^L_S \mathcal M_2)&\overset{\quad\qquad}{=}(f_1\times{\rm id})^!(\mathcal M_1\boxtimes^L_S R\mathcal Hom(\mathcal L_2, q_2^!\Lambda_S))\\
\nonumber&\overset{\eqref{eq:prop:fUpperIsomKun101}}{\cong} (f_1\times{\rm id})^!(R\mathcal Hom({\rm pr}_2^\ast\mathcal L_2, {\rm pr}_1^!\mathcal M_1))\\
&\overset{\eqref{eq:pullbackupp}}{\cong} R\mathcal Hom((f_1\times {\rm id})^*{\rm pr}_2^*\mathcal L_2, (f_1\times {\rm id})^!{\rm pr}_1^!\mathcal M_1)\\
\nonumber&\overset{\quad\qquad}{\cong} R\mathcal Hom((f_1\times {\rm id})^*{\rm pr}_2^*\mathcal L_2, {\rm pr}_1^!f_1^!\mathcal M_1)\\
\nonumber&\overset{\eqref{eq:prop:fUpperIsomKun102}}{\cong} f_1^!\mathcal M_1\boxtimes^L_S R\mathcal Hom (\mathcal L_2, q_2^!\Lambda_S) \cong f_1^!\mathcal M_1\boxtimes^L_S \mathcal M_2.
\end{align}
This finishes the proof.
\end{proof}

\subsection{Relative cohomological characteristic class}
\subsubsection{} We introduce some notation for convenience. For any commutative diagram 
 \[
 \xymatrix{
 W\ar[rd]_-f\ar[rr]^-h&& V\ar[ld]^-g\\
&{\rm Spec}k&
 }
 \]
 of schemes, we put
\begin{align}
&\mathcal K_W\coloneq Rf^!\Lambda,\quad \\
&\mathcal K_{W/V}\coloneq Rh^!\Lambda_{V}.
\end{align}
Under the notation in \ref{Sec4:notation}, by Proposition \ref{prop:fUpperIsomKun},
we have an isomorphism
\begin{align}\label{eq:KWViso}
\mathcal K_{X_1/S}\boxtimes^L_S \mathcal K_{X_2/S}\simeq \mathcal K_{X/S}.
\end{align}
\subsubsection{}Consider a cartesian diagram 
\begin{align}
\begin{gathered}
\xymatrix{
E\ar[rd]^-e\ar[d]\ar[r]&D\ar[d]^-d\\
C\ar[r]^-c&X
}
\end{gathered}
\end{align}
of schemes over $k$.
Let $\mathcal F$, $\mathcal G$ and $\mathcal H$ be objects of $D_c^b(X,\Lambda)$ and $\mathcal F\otimes\mathcal G\rightarrow \mathcal H$ any morphism. By the K\"unneth isomorphism \cite[XVII, Th\'eor\`eme 5.4.3]{SGA4T3} and adjunction, we have 
\[
e_!(c^!\mathcal F\boxtimes^L_X d^! \mathcal G)\xrightarrow{\simeq} c_!c^!\mathcal F\otimes^L d_!d^!\mathcal G\rightarrow \mathcal F\otimes \mathcal G\rightarrow \mathcal H.
\] 
By adjunction, we get a morphism 
\begin{align}\label{eq:pair1-1}
c^!\mathcal F\boxtimes^L_X d^! \mathcal G\rightarrow e^!\mathcal H.
\end{align}
Thus we get a pairing 
\begin{align}\label{eq:pair1}
\langle,\rangle: H^0(C,c^!\mathcal F)\otimes H^0(D,d^!\mathcal G)\rightarrow H^0(E, e^!\mathcal H).
\end{align}

\subsubsection{}
Now we define the relative Verdier pairing by  applying the map \eqref{eq:pair1} to relative cohomological correspondences.
Let  $\pi_1\colon X_1\rightarrow S$ and $\pi_2\colon X_2\rightarrow S$  be  smooth morphisms. Consider a cartesian diagram 
\begin{align}
\begin{gathered}
\xymatrix@C=5pc{
E\ar[rd]^-e\ar[d]\ar[r]&D\ar[d]^-{d=(d_1,d_2)}\\
C\ar[r]_-{c=(c_1,c_2)}&X=X_1\times_S X_2
}
\end{gathered}
\end{align}
of schemes over $S$. Let $\mathcal F_1\in D_c^b(X_1,\Lambda)$ and $\mathcal F_2\in D_c^b(X_2,\Lambda)$. 
Assume that one of the following conditions holds:
\begin{enumerate}
\item $\pi_1\colon X_1\to S$ is $SS(\mathcal F_1, X_1/k)$-transversal;
\item $\pi_2\colon X_2\to S$ is $SS(\mathcal F_2, X_2/k)$-transversal.
\end{enumerate}
By Corollary \ref{cor:SSKunn}, we have 
\begin{align}
\nonumber&R\mathcal Hom({\rm pr}_2^\ast\mathcal F_2,{\rm pr}_1^!\mathcal F_1)\otimes^L R\mathcal Hom({\rm pr}_1^\ast\mathcal F_1,{\rm pr}_2^!\mathcal F_2)\\
\label{eq:pair3}&\xrightarrow{\simeq}(\mathcal F_1\boxtimes^L_S R\mathcal Hom(\mathcal F_2, \pi_2^!\Lambda_S))\otimes^L( R\mathcal Hom(\mathcal F_1, \pi_1^!\Lambda_S)\boxtimes^L_S\mathcal F_2)\\
\nonumber&\xrightarrow{\rm evaluation} \pi_1^!\Lambda_S\boxtimes^L_S \pi_2^!\Lambda_S\overset{\eqref{eq:KWViso}}{\cong}\mathcal K_{X/S}. 
\end{align}
By \eqref{eq:pullbackupp}, \eqref{eq:pair1-1}, \eqref{eq:pair1} and \eqref{eq:pair3}, we get the following pairings 
\begin{align}
\label{eq:pair4-1}&c_!R\mathcal Hom({c}_2^\ast\mathcal F_2,{c}_1^!\mathcal F_1)\otimes^L d_!R\mathcal Hom({d}_1^\ast\mathcal F_1,{d}_2^!\mathcal F_2)\rightarrow  e_!\mathcal K_{E/S},\\
\label{eq:pair4}
&\langle,\rangle:
Hom({c}_2^\ast\mathcal F_2,{c}_1^!\mathcal F_1)\otimes Hom({d}_1^\ast\mathcal F_1,{d}_2^!\mathcal F_2)\rightarrow H^0(E, e^!(\mathcal K_{X/S}))=H^0(E, \mathcal K_{E/S}).
\end{align}
The pairing \eqref{eq:pair4} is called the \emph{relative Verdier pairing} (cf. \cite[Expos\'e III (4.2.5)]{Gro77}).

\begin{definition}\label{def:ccc}
Let $f\colon X\rightarrow S$ be a smooth morphism purely of relative dimension $n$ and $\mathcal F\in D_c^b(X,\Lambda)$. We assume that $f$ is $SS(\mathcal F, X/k)$-transversal.
Let $c=(c_1,c_2)\colon C\rightarrow X\times_SX$ be a closed immersion and $u\colon c_2^\ast\mathcal F\rightarrow c_1^!\mathcal F$ be a relative cohomological correspondence on $C$.
We define the \emph{relative cohomological characteristic class} $ccc_{X/S}(u)$ of $u$ to be the 
cohomology class $\langle u,1\rangle\in H^0_{C\cap X}(X, \mathcal K_{X/S})$ defined by the pairing \eqref{eq:pair4}.

In particular, if $C=X$ and $c\colon C\to X\times_SX$ is the diagonal and if $u\colon \mathcal F\rightarrow \mathcal F$ is the identity, 
we write 
\[ccc_{X/S}(\mathcal F)=\langle1,1\rangle \quad{\rm in}\quad  H^{2n}(X, \Lambda(n))\]
 and call it the \emph{relative cohomological characteristic class} of $\mathcal F$. 
\end{definition}
If $S$ is the spectrum of a perfect field, then the above definition is  \cite[Definition 2.1.1]{AS07}.
\begin{example}
 If $\mathcal F$ is a  locally constant and constructible sheaf of $\Lambda$-modules on $X$, then we have $ccc_{X/S}\mathcal F={\rm rank}\mathcal F\cdot c_n(\Omega^\vee_{X/S})\cap [X]\in CH^n(X)$.
\end{example}

\begin{conjecture}\label{conj:ccequality}
Let $S$ be a smooth connected scheme over a perfect field  $k$ of characteristic $p$.
Let $f\colon X\to S$ be a smooth morphism purely of relatively dimension $n$ and $\mathcal F\in D_c^b(X,\Lambda)$. Assume that
$f$ is $SS(\mathcal F, X/k)$-transversal. Let ${\rm cl}\colon {\rm CH}^n(X)\rightarrow  H^{2n}(X, \Lambda(n))$ be the cycle class map. Then we have
\begin{equation}
{\rm cl}(cc_{X/S}(\mathcal F))=ccc_{X/S}(\mathcal F) \quad{\rm in}\quad  H^{2n}(X, \Lambda(n)),
\end{equation}
where $cc_{X/S}(\mathcal F)$ is the relative characteristic class defined in Definition \ref{def:rcclass}. 
\end{conjecture}
If $S$ is the spectrum of a perfect field, then the above conjecture  is \cite[ Conjecture 6.8.1]{Sai16}.

\subsection{Proper push-forward of relative cohomological characteristic class}
\subsubsection{}
For $i=1,2$, let $f_i\colon X_i\to Y_i$ be a proper morphism between smooth schemes over $S$.
Let $X\coloneq X_1\times_SX_2$, $Y\coloneq Y_1\times_S Y_2$ and $f\coloneq f_1\times_Sf_2$. 
Let $p_i\colon X\to X_i $ and $ q_i\colon Y\to Y_i$ be the canonical projections for $i=1,2$.
Consider a commutative diagram 
\begin{align}
\begin{gathered}
\xymatrix{
X\ar[d]_-f& C\ar[l]_-c\ar[d]^-g\\
Y&D\ar[l]_-d
}
\end{gathered}
\end{align}
of schemes over $S$.
Assume that $c$ is proper. 
Put $c_i=p_i c$ and $d_i=q_id$.
By \cite[Construction 7.17]{Zh15}, we have the following push-forward maps for cohomological correspondence (see also \cite[Expos\'e III, (3.7.6)]{Gro77} if $S$ is the spectrum of a field):
\begin{align}
\label{eq:pfccor}
&f_\ast \colon Hom(c_2^\ast \mathcal L_2, c_1^!\mathcal L_1)\to Hom(d_2^\ast(f_{2!}\mathcal L_2), d_1^!(f_{1\ast}\mathcal L_1) ),\\
\label{eq:pfccor2}
&f_\ast \colon g_\ast R\mathcal Hom(c_2^\ast \mathcal L_2, c_1^!\mathcal L_1)\to R\mathcal Hom(d_2^\ast(f_{2!}\mathcal L_2), d_1^!(f_{1\ast}\mathcal L_1) ).
\end{align}

\begin{theorem}[{\cite[Th\'eor\`eme 4.4]{Gro77}}]\label{thm:mainPSILL}
For $i=1,2$, let $f_i\colon X_i\to Y_i$ be a proper morphism between smooth schemes over $S$.
Let $X\coloneq X_1\times_SX_2$, $Y\coloneq Y_1\times_S Y_2$ and $f\coloneq f_1\times_Sf_2$. 
Let $p_i\colon X\to X_i $ and $ q_i\colon Y\to Y_i$ be the canonical projections for $i=1,2$.
Consider the following commutative diagram with cartesian horizontal faces

\[
\begin{tikzcd}[row sep=2.5em]
C'  \arrow[dr,swap,"c'"] \arrow[dd,swap,"f'"] &&
  C \arrow[dd,swap,"g" near start] \arrow[dr] \arrow[ll]\arrow[dl, swap,"c"]\\
& X  &&
  C'' \arrow[dd,"f''"]\arrow[ll,crossing over, swap, "c''" near start] \\
D'  \arrow[dr,swap,"d'"] && D \arrow[dr] \arrow[ll] \ar[dl,"d"] \\
& Y  \arrow[uu,<-,crossing over,"f" near end]&& D''\arrow[ll,"d''"]
\end{tikzcd}
\]
where $c', c'', d'$ and $ d''$ are proper morphisms between smooth schemes over $S$. Let $c_i'=p_i c', c_i''=p_i c'', d_i'=q_i d', d_i''=q_i d''$ for $i=1,2$. Let $\mathcal L_i \in D^b_c(X_i,\Lambda)$ and we put $\mathcal M_i=f_{i\ast}\mathcal L_i$ for $i=1,2$. 
Assume that one of the following conditions holds:
\begin{enumerate}
\item $ X_1\to S$ is $SS(\mathcal L_1, X_1/k)$-transversal;
\item $ X_2\to S$ is $SS(\mathcal L_2, X_2/k)$-transversal.
\end{enumerate}
Then we have the following commutative diagram
\begin{align}\label{eq:thm:mainPSILL} 
\begin{gathered}
\xymatrix{
f_*{c}^\prime_*R\mathcal Hom({c}^{\prime*}_2\mathcal L_2, {c}^{\prime!}_1\mathcal L_1)\otimes^Lf_*{c}^{\prime\prime}_*R\mathcal Hom({c}^{\prime\prime*}_1\mathcal L_1, {c}^{\prime\prime!}_2\mathcal L_2) \ar[r]^-{(1)}\ar[d]^{(2)} & f_*c_*\mathcal K_{C/S}\ar[d]^{(4)} \\
{d}^\prime_*R\mathcal Hom({d}_2^{\prime*}\mathcal M_2, {d}_1^{\prime!}\mathcal M_1)\otimes^L{d}^{\prime\prime}_*R\mathcal Hom({d}_1^{\prime\prime*}\mathcal M_1,{d}_2^{\prime\prime!}\mathcal M_2) \ar[r]^-{(3)} &d_*\mathcal K_{D/S}
}
\end{gathered}
\end{align}
where $(3)$ is given by \eqref{eq:pair4-1}, $(1)$ is the composition of $f_\ast(\eqref{eq:pair4-1})$ with the canonical map $f_\ast c^{\prime}_\ast\otimes^L f_\ast c_\ast^{\prime\prime}\to f_\ast(c_\ast^\prime\otimes c_\ast^{\prime\prime})$, $(2)$ is induced from \eqref{eq:pfccor2}, and $(4)$ is defined by
\begin{align}
f_*c_*\mathcal K_{C/S}\simeq d_\ast g_\ast \mathcal K_{C/S}=d_\ast g_!g^!\mathcal K_{D/S}\xrightarrow{\rm adj}d_\ast\mathcal K_{D/S}.
\end{align}
\end{theorem}
If $S$ is the spectrum of a field, this is proved in \cite[Th\'eore\`eme 4.4]{Gro77}. We use the same notation as {\emph{loc.cit.}}
\begin{proof}
By \cite[Lemma 3.8 and  Lemma 4.2.6]{Sai16} and the assumption, one of the following conditions holds:
\begin{enumerate}
\item[(a1)] $ Y_1\to S$ is $SS(\mathcal M_1, Y_1/k)$-transversal;
\item[(a2)] $ Y_2\to S$ is $SS(\mathcal M_2, Y_2/k)$-transversal.
\end{enumerate}
Now we can use the same  proof of \cite[Th\'eor\`eme 4.4]{Gro77}.
We only sketch the main step. Put 
\begin{align}
&\mathcal P=\mathcal L_1\boxtimes^L_S R\mathcal Hom(\mathcal L_2,\mathcal K_{X_2/S}),\quad \mathcal Q=R\mathcal Hom(\mathcal L_1,\mathcal K_{X_1/S})\boxtimes^L_S \mathcal L_2\\
&\mathcal E=\mathcal M_1 \boxtimes^L_SR\mathcal Hom(\mathcal M_2,\mathcal K_{Y_2/S}),\quad
\mathcal F=R\mathcal Hom(\mathcal M_1, \mathcal K_{Y_1/S})\boxtimes_S^L \mathcal M_2.
\end{align}
Then the theorem follows from the  following commutative  diagram
\[
\xymatrix{
f_\ast c_\ast^\prime c^{\prime!}\mathcal P\otimes^L f_\ast c_\ast^{\prime\prime} c^{\prime\prime!}\mathcal Q\ar[rr]\ar[d]&&f_\ast c_\ast c^!(\mathcal P\otimes^L\mathcal Q)\ar[d]\ar[r]&f_\ast c_\ast c^!\mathcal K_{X/S}\ar[d]\\
d_\ast^\prime d^{\prime!}f_\ast\mathcal P\otimes^L d_\ast^{\prime\prime}d^{\prime\prime!}f_\ast\mathcal Q\ar[d]\ar[r]&d_\ast d^!(f_\ast\mathcal P\otimes^L f_\ast\mathcal Q)\ar[r]\ar[d]&
d_\ast d^! f_\ast(\mathcal P\otimes^L\mathcal Q)\ar[r]&d_\ast d^!f_\ast\mathcal K_{X/S}\ar[d]\\
d_\ast^\prime d^{\prime!}\mathcal E\otimes^L d_\ast^{\prime\prime}d^{\prime\prime!}\mathcal F
\ar[r]&d_\ast d^!(\mathcal E\otimes^L\mathcal F)\ar[rr]&& d_\ast d^!\mathcal K_{Y/S}
}
\]
where commutativity can be verified following the same argument of \cite[Th\'eor\`eme 4.4]{Gro77}.
\end{proof}
\begin{corollary}[{\cite[Corollaire 4.5]{Gro77}}]\label{cor:mainPSILL}
Under the assumptions of Theorem \ref{thm:mainPSILL}, we have a commutative diagram
\begin{align}
\xymatrix{
Hom(c_2^{\prime\ast} \mathcal L_2, c_1^{\prime!}\mathcal L_1)\otimes Hom(c_1^{\prime\prime\ast} \mathcal L_1, c_2^{\prime\prime!}\mathcal L_2)\ar[d]_-{\eqref{eq:pfccor}\otimes\eqref{eq:pfccor}}\ar[r]&H^0(C,\mathcal K_{C/S})\ar[d]^-{g_\ast}\\
Hom(d_2^{\prime\ast} f_{2\ast}\mathcal L_2, d_1^{\prime!}f_{1\ast}\mathcal L_1)\otimes Hom(d_1^{\prime\prime\ast} f_{1\ast}\mathcal L_1, d_2^{\prime\prime!}f_{2\ast}\mathcal L_2)\ar[r]&H^0(D,\mathcal K_{D/S}).
}
\end{align}
\end{corollary}
\begin{corollary}\label{cor:pushccc}
Let $S$ be a smooth connected  scheme over a perfect  field $k$ of characteristic $p$.
Let $f\colon X\rightarrow S$ be a smooth morphism purely of relative dimension $n$ and  $g\colon Y\to S$ a smooth morphism purely of relative dimension $m$.
Assume that $f$ is $SS(\mathcal F, X/k)$-transversal. 
Then for any proper morphism $h\colon X\to Y$ over $S$, 
\begin{align}\label{eq:cor:pushccc1}
\begin{gathered}
\xymatrix{
X\ar[rr]^-h\ar[rd]_-f&&Y\ar[ld]^-g\\
&S
}
\end{gathered}
\end{align}
we have 
\begin{align}\label{eq:cor:pushccc2}
f_\ast ccc_{X/S}(\mathcal F)= ccc_{Y/S}(Rf_\ast\mathcal F)\quad {\rm in}\quad H^{2m}(Y,\Lambda(m)).
\end{align}
\end{corollary}
\begin{proof}
This follows from Corollary \ref{cor:mainPSILL} and Definition \ref{def:ccc}.
\end{proof}

\end{document}